\documentclass[12pt]{amsart}

\usepackage{amsmath,graphicx}
\usepackage{adjustbox}
\usepackage{amssymb}
\usepackage{tikz, tikz-cd}
\usepackage{fullpage}
\usepackage[all]{xy}
\usepackage[margin=1in]{geometry}
\usepackage{amsthm}
\usepackage{amsfonts}
\usepackage{bbm}
\usepackage{comment}
\usepackage{amsmath}
\usepackage{amsthm,diagbox}
\usepackage{thm-restate}

\usepackage{subcaption}

\usepackage{array} 
\usepackage{color}
\usepackage[utf8]{inputenc}
\usepackage[english]{babel}
\usepackage{hyperref}
\newcolumntype{L}{>{$}l<{$}}
\usepackage{bm}

\DeclareMathSymbol{\shortminus}{\mathbin}{AMSa}{"39}

\newtheorem{theorem}{Theorem}[section]
\newtheorem{lemma}[theorem]{Lemma}

\newtheorem{conjecture}[theorem]{Conjecture}
\newtheorem{proposition}[theorem]{Proposition}
\theoremstyle{definition}  
\newtheorem{definition} [theorem] {Definition} 

\newtheorem{example} [theorem] {Example}
\newtheorem{remark} [theorem] {Remark}
\newtheorem{question} [theorem] {Question}

\theoremstyle{definition}

\newcommand{\D}{{\mathbb{D}}}

\newcommand{\Q}{{\mathbb{Q}}}

\newcommand{\R}{{\mathbb{R}}}
\newcommand{\Z}{{\mathbb{Z}}}

\newcommand{\im}{\text{im}}

\newcommand{\surj}{\twoheadrightarrow}

\newcommand{\PSL}{\text{PSL}}

\newcommand{\spin}{\text{spin}}
\newcommand{\Sym}{\text{Sym}}

\newcommand{\T}{\bm{T}}
\newcommand{\s}{\mathfrak{s}}

\DeclareMathOperator{\HFK}{HFK}

\DeclareMathOperator{\HF}{HF}

\DeclareMathOperator{\SFH}{SFH}

\DeclareMathOperator{\CF}{CF}
\DeclareMathOperator{\CFK}{CFK}

\DeclareMathOperator{\rank}{rank}

\DeclareMathOperator{\sing}{sing}
\DeclareMathOperator{\codim}{codim}
\DeclareMathOperator{\ind}{ind}

\usepackage[pagewise]{lineno}

\title{The next-to-top term of the knot Floer homology of some non-fibered knots}
\author[Binns]{Fraser Binns}
\author[Wan]{Shunyu Wan}
\address[]{Department of Mathematics, Princeton University}
\email{fb1673@princeton.edu}
\address[]{Department of Mathematics, Georgia Institute of Technology}
\email{swan48@gatech.edu}
\date{\today}
\thanks{FB was supported by the Simons Grant {\em New structures in low-dimensional topology}. SW was supported by the Georgia Tech Postdoc funding.}
\keywords{Knot Floer homology}
\subjclass{57K18}

\begin{document}

\begin{abstract}
 Sivek conjectured that the rank of knot Floer homology in the next-to-top Alexander grading is at least the rank in the top Alexander grading. Baldwin and Vela-Vick verified this conjecture in the case of fibered knots~\cite{baldwin_note_2018}. Ni gave a generalization of this result (for knots in generalized $L$-spaces) to  cases in which the knot Floer homology satisfies an algebraic condition~\cite{Ninexttotop}. We give an independent generalization of Baldwin and Vela-Vick's result to a family of knots with Seifert surfaces satisfying certain conditions. 
\end{abstract}
\maketitle

\section{Introduction}

 Knot Floer homology is a powerful knot invariant due independently to Ozsv\'ath-Szab\'o~\cite{Holomorphicdisksandknotinvariants} and J.Rasmussen~\cite{Rasmussen}, taking value in the category of bigraded vector spaces. In this paper we will be interested in a special case of the following general question:

\begin{question}[The Geography Question]
    Which bigraded vector spaces arise as the knot Floer homology of a knot?
\end{question}

The following is a conjectural constraint on the geography of knot Floer homology: 

\begin{conjecture}[Sivek]\label{con:sivek}
   If $K$ is a non-trivial null-homologous knot in a $3$-manifold $Y$, then \begin{align*}
    \rank(\widehat{\HFK}(K,Y,g(K)-1))\geq \rank(\widehat{\HFK}(K,Y,g(K))).\end{align*}
\end{conjecture}

The $Y=S^3$ case of this conjecture is posed as a question in~\cite[Question 1.12]{baldwin_note_2018}. Note that the statement is false in the $g(K)=0$ case, since then $\widehat{\HFK}(K,Y)$ is supported in a single Alexander grading.

There has been some progress towards proving Conjecture~\ref{con:sivek}. Notably, it was verified by Ni under the hypothesis that $\widehat{\HFK}(K,Y,g(K))$ is supported in a single mod 2 Maslov grading and $Y$ is a \emph{generalized $L$-space}~\cite[Theorem 1.1]{Ninexttotop}. This can be viewed as an extension of a special case of a result of Baldwin and Vela-Vick \cite{baldwin_note_2018}, who verified the conjecture in the case that $K$ is fibered, without any additional hypotheses on $Y$ (here $\rank(\widehat\HFK(K,Y,g(K))=1$ by~\cite{ni2007knot,ghiggini2008knot}, whence in turn $\widehat\HFK(K,Y,g(K))$ is homogeneous). In this paper we generalize Baldwin and Vela-Vick's techniques to the setting of knots with \emph{$w$-avoiding exterior}. Informally, knots with $w$-avoiding exterior are knots admitting Seifert surfaces with relatively simple exteriors.

\begin{theorem}\label{thm:mainnoMaslov}
    If $K$ is a non-trivial knot with $w$-avoiding exterior, then ${\widehat{\HFK}(K,Y,g(K)-1)}$ contains a $\widehat{\HFK}(K,Y,g(K))[-1]$ summand.
\end{theorem}

Here $[k]$ denotes an upward shift in the Maslov grading by $k$. To define knots with $w$-avoiding exterior, we require the following preliminary Heegaard-diagrammatic definition: 

\begin{definition}
   A knot is \emph{$w$-avoiding} if it admits an admissible Heegaard diagram $(\Sigma,\bm{\alpha},\bm{\beta},z,w)$ such that every Maslov index $1$ disk has multiplicity zero at $w$.
\end{definition}

 See Section~\ref{subsec:knotFloer} for the definition of admissible. Examples of such knots include the cores of any surgery on the unknot. Other examples can be obtained from \emph{strong $L$-spaces}, as introduced in~\cite{MR3091604}. These are $L$-spaces admitting Heegaard diagrams, $\mathcal{H}=(\Sigma,\bm{\alpha},\bm{\beta},z)$, such that $\T_{\bm{\alpha}}\cap\T_{\bm{\beta}}=|\widehat{\HF}(Y)|$. Examples include double branched covers of non-split alternating links~\cite[Corollary 3.6]{MR3065184}. To obtain a $w$-avoiding knot from a Heegaard diagram for a strong $L$-space, one adds a basepoint $w$ to $\Sigma$. There can be no Maslov index $1$ disk with non-zero multiplicity at $w$, since each generator lies in a distinct $\spin^c$-structure, and so cannot be connected by disks of any index.

\begin{remark}
    As pointed out to the authors by Yi Ni, it is straightforward to check that null-homologous $w$-avoiding knots in L-spaces satisfy Conjecture \ref{con:sivek}, using the arguments he gives in~\cite{Ninexttotop}.
\end{remark}

We can now clarify the new terminology in the statement of Theorem~\ref{thm:mainnoMaslov}.

\begin{definition}\label{def:avoidext}

 A knot $K$ has \emph{$w$-avoiding exterior} if $K$ has a minimal genus Seifert surface whose sutured exterior admits a decomposition along a pair of product annuli to a product sutured manifold and the sutured exterior of some $w$-avoiding knot. 

\end{definition}

See Section~\ref{subsec:suturedmanifolds} and Remark~\ref{rem:generalhard} for relevant definitions and further discussion. 

Our result is a strict generalization of Baldwin and Vela-Vick's result that fibered knots satisfy Sivek's conjecture~\cite{baldwin_note_2018}. Non-fibered examples of knots with $w$-avoiding exterior include $5_2$ in $S^3$; see Example~\ref{ex:52}. This particular example also satisfies Ni's algebraic condition and is a knot in a generalized $L$-space. However, if $K$ has $w$-avoiding exterior, one can check that the core of $\frac{1}{n}$-surgery on $K$ has a minimal genus Seifert surface exterior diffeomorphic to the exterior of a minimal genus Seifert surface for $K$. Since, for example, $\frac{1}{n}$-surgery on $5_2$ is not a generalized $L$-space --- which can be seen using the rational surgery formula~\cite{ozsvath2010knotrationalsurgeries}, or its reinterpretation in terms of immersed curves~\cite{hanselman2016bordered} --- we see that Theorem~\ref{thm:mainnoMaslov} gives examples of knots satisfying Sivek's conjecture that are not covered by Ni's theorem.

The idea of the proof of Theorem~\ref{thm:mainnoMaslov} is concrete; we work at the level of Heegaard diagrams and explicitly exhibit an appropriate number of generators of the relevant bigradings. The first step in this process is to find generators of the minimal Alexander grading, extending work of Vela-Vick who found generators of the minimal Alexander grading of fibered knots~\cite{vela2011transverse}, and then applying a similar trick to that given in~\cite{baldwin_note_2018} we find generators in the next-to-bottom Alexander grading. 

The outline of the paper is as follows. In Section~\ref{sec:background} we review relevant background material on Heegaard Floer homology. In Section~\ref{sec:Heegaarddiagrams} we discuss the Heegaard diagrams we will use for the rest of the paper. In Section~\ref{sec:minimalgrading} we find generators of the bottom grading of knot Floer homology. In Section~\ref{sec:nexttominimal} we find generators of the next-to-bottom Alexander grading.

\subsection*{Acknowledgments}
We would like to thank the organizers of the fall 2023 AMS central sectional meeting, where we began working on this project. The first author would like to thank Diana Hubbard for a number of helpful conversations on a related project, some technical aspects of which overlap with those of the current paper. He would also like to thank Gary Guth, Gheehyun Nahm, Robert Lipshitz, Peter Ozsv\'ath, Zolt\'an Szab\'o, and Yonghan Xiao for other useful conversations. The second author would like to thank Tom Mark for useful conversations.
\section{Background}\label{sec:background}

In this Section we survey background material that we will use in subsequent sections. In Section~\ref{subsec:knotFloer} we briefly review aspects of knot Floer homology, in Section~\ref{subsec:suturedmanifolds} we briefly review sutured manifolds, and in Section~\ref{sec:SFH} we briefly review sutured Floer homology. Throughout this paper we work with null-homologous knots. 
\subsection{Knot Floer homology}\label{subsec:knotFloer}

Knot Floer homology is a knot invariant due independently to Ozsv\'ath and Szab\'o~\cite{ozsvath2005knot} and J. Rasmussen~\cite{Rasmussen}. The knot Floer homology of $K$ is denoted by $\widehat{\HFK}(K)$.

We will use a modified version of knot Floer homology, corresponding to knots decorated with an even number of basepoints. A knot decorated with $2n>0$-basepoints can be encoded by a multi-pointed Heegaard diagram $\mathcal{H}=(\Sigma,\bm{\alpha},\bm{\beta},\bm{z},\bm{w})$. For reasons that will become apparent in Section~\ref{sec:Heegaarddiagrams} and Section~\ref{sec:minimalgrading}, we will be primarily interested in the case in which $n=2g(K)$. We will also assume that all of our Heegaard diagrams are \emph{admissible}; that is every \emph{periodic domain} with multiplicity $0$ at each $z$-basepoint or boundary component of $\Sigma$ has positive and negative coefficients. Knot Floer homology then assigns to each such Heegaard diagram a finitely generated chain complex over $\Z/2$, $\widetilde{\CFK}(\mathcal{H})$, equipped with a differential $\partial_{\widetilde{\CFK}}$ that counts pseudo-holomorphic disks in an auxiliary symplectic manifold. The homology of this complex is denoted by $\widetilde{\HFK}(\mathcal{H})$ or $\widetilde{\HFK}(K,n)$ and is an invariant of the pointed knot. In the case that $n=1$ we recover the original version of knot Floer homology --- i.e. we have that 
\begin{align*}\widetilde{\HFK}(K,1)\cong\widehat{\HFK}(K).\end{align*}
\label{eq:tildehat}

$\widetilde{\HFK}(K,n)$ comes equipped with a $\Q$-valued grading called the \emph{Maslov grading}. A choice of relative homology class of Seifert surface $S$ for $K$ induces a grading $A_{S}$ on $\widetilde{\HFK}(K,n)$ called the \emph{Alexander grading for $[-S]$}. We will sometimes leave the choice of Seifert surface implicit.

The dependence of $\widetilde{\HFK}(K,n)$ on $n$ is straightforward; 
\begin{align}
\widetilde{\HFK}(K,n+1)\cong\widetilde{\HFK}(K,n)\otimes V.\end{align}\label{eq:tilden} Here $V$ is a rank two vector space supported in $(A,m)$ gradings $(\frac{1}{2},0)$ and $(-\frac{1}{2},-1)$.

If $K$ is a null-homologous knot decorated with $2n$ basepoints then there is a spectral sequence from $\widetilde{\CFK}(K,n)$ with $E_2$ page $\widetilde{\HFK}(K,n)$ and $E_\infty$ page $\widehat{\HF}(Y)\otimes V^{\otimes(n-1)}$. Here $\widehat{\HF}(Y)$ is the \emph{Heegaard Floer homology} of $Y$, an invariant introduced by Ozsv\'ath and Szab\'o~\cite{ozsvath2004holomorphic} prior to knot Floer homology.

$\widetilde{\HFK}(K,n)$ has a number of useful properties. For example, the maximum non-trivial Alexander grading in which $\widetilde{\HFK}(K,n)$ is non-trivial, $A_\text{max}$ is given by $g(K)+\dfrac{n-1}{2}$~\cite{OSgenusbounds}. Moreover, $K$ is fibered if and only if $\widetilde{\HFK}(K,n)$ has rank one in Alexander grading $g(K)+\dfrac{n-1}{2}$~\cite{ni2007knot,ghiggini2008knot}.

In this paper we will define maps on knot Floer homology groups. It is common to define such maps using counts of pseudo-holomorphic triangles. Given $\bm{\alpha},\bm{\beta},\bm{\gamma}$ collection of $g$ curves in a surface of genus $g$, $\Sigma$, satisfying appropriate non-degeneracy assumptions, we have a chain map $$f:\widehat{\CFK}(\mathcal{H}_{\bm{\alpha},\bm{\beta}})\otimes\widehat{\CFK}(\mathcal{H}_{\bm{\beta},\bm{\gamma}})\to\widehat{\CFK}(\mathcal{H}_{\bm{\alpha}\bm{,\gamma}}).$$

\noindent This is given by counting holomorphic disks subject to the boundary conditions indicated in Figure~\ref{fig:triangle}.

\begin{figure}[htbp]
    \centering
    \begin{tikzpicture}
        \coordinate (A) at (0,0);
        \coordinate (B) at (2,0);
        \coordinate (C) at (1,{sqrt(3)});
        
        \draw[red, thick] (A) -- (B) node[midway, below] {$\alpha$};
        \draw[blue, thick] (B) -- (C) node[midway, right] {$\beta$};
        \draw[green, thick] (C) -- (A) node[midway, left] {$\gamma$};   
        \node[below] at (A) {};
        \node[below] at (B) {};
        \node[above] at (C) {$\bm{\theta}$};
    \end{tikzpicture}
    \caption{ $f$ and $f_{\bm{\theta}}$ count triangles of this form.}
    \label{fig:triangle}
\end{figure}
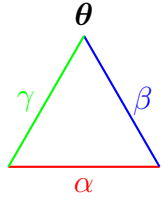
\noindent Hence if we have a cycle $\bm{\theta}\in \widehat{\CFK}(\mathcal{H}_{\bm{\beta},\bm{\gamma}})$ we obtain a map $f_\theta:\widehat{\CFK}(\mathcal{H}_{\bm{\alpha},\bm{\beta}})\to\widehat{\CFK}(\mathcal{H}_{\bm{\alpha},\bm{\gamma}})$ defined by $f_{\bm{\theta}}:\bm{x}\to f(\bm{x}\otimes \bm{\theta})$. Note that since we will work with Heegaard diagrams where we reverse the roles of $\alpha$ and $\beta$-curves, our counts of holomorphic disks will look different in practice.

\subsection{Sutured Manifolds}\label{subsec:suturedmanifolds}
A \emph{sutured manifold},  $(Y,\gamma)$, is a $3$-manifold $Y$ equipped with a decomposition of its boundary, $\gamma$. See~\cite[Definition 2.6]{gabai1983foliations} for details. These manifolds were first studied by Gabai in the context of taut foliations~\cite{gabai1983foliations}. In this paper we will only be interested in a subset of manifolds called \emph{balanced sutured manifolds}, as defined in~\cite[Definition 2.2]{juhasz2006holomorphic}. We will typically suppress the term ``balanced" henceforth.

If $K$ is a knot in a $3$-manifold $M$, then the exterior of $K$ in $M$ --- which we denote $Y(K)$ --- carries a natural sutured structure $(Y(K),\gamma(K,n))$ where $\gamma(K,n)$ consists of $2n$ appropriately oriented meridians. Likewise, if $S$ is a Seifert surface for a knot $K$ in $M$ then the exterior of $S$, which we denote $Y(S)$, carries a natural sutured structure $(Y(S),\gamma(S))$, where $\gamma(S)$ consists of a single push-off of $K=\partial S$ into $\partial(Y(S))$.

A \emph{product annulus} in a sutured manifold $(Y,\gamma)$ is a properly embedded annulus in $Y$ with one boundary component in $R_+(\gamma)$ and the other in $R_-(\gamma)$. Given a product annulus $A$ in a sutured manifold $(Y,\gamma)$, one can obtain a new sutured manifold $(Y',\gamma')$ where $Y'=Y\setminus\nu(A)$ and $\gamma'=\gamma\cup \partial_\pm(\nu(A))$, where $\partial_\pm(\nu(A))$ are the pair of push-offs of $A$ in $\partial Y'$.

\begin{example}\label{ex:52}
    In the introduction, we claimed that $5_2$ has $w$-avoiding exterior. To see this, first recall that the \emph{guts} of a knot $K$ are the non-product components of the sutured manifold obtained by decomposing the sutured exterior of $K$ along a maximal family of disjoint minimal genus Seifert surfaces, and then decomposing that sutured manifold along a maximal family of incompressible, non-boundary parallel decomposing annuli. The guts of $5_2$ are the sutured exterior of the core of $2$-surgery on the unknot~\cite{baldwin2022floer,li2022seifert}. This is a $w$-avoiding knot; it admits an admissible genus one Heegaard diagram with two intersection points and no disks. In particular, $5_2$ has a minimal genus Seifert surface whose sutured exterior admits a decomposition along a pair of product annuli into the sutured exterior of the core of two-surgery on the unknot and a product sutured manifold $(P\times[0,1],\partial P\times[0,1])$, where  $P$ is a pair of pants.
\end{example} 

Sutured manifolds can be thought of as $3$-manifolds which admit \emph{sutured Heegaard diagrams}. A sutured Heegaard diagram consists of a surface with boundary $\Sigma$, a collection of $n$ homologically linearly independent curves $\bm{\alpha}$ and another collection of $n$ homologically linearly independent curves, $\bm{\beta}$. Given $(\Sigma,\bm{\alpha},\bm{\beta})$, one can obtain $M$ by attaching disks along their boundaries to each of the $\bm{\alpha}$ and $\bm{\beta}$-curves and then thickening the resulting complex. $\gamma$ is then given by a neighborhood of the boundary of $\Sigma$ in the boundary of $M$.

\subsection{Sutured Floer Homology}\label{sec:SFH}
Juh\'asz generalized knot Floer homology to \emph{sutured Floer homology}, an invariant of balanced sutured manifolds~\cite{juhasz2006holomorphic}. Sutured Floer homology has underlying chain complex $\CF(Y,\gamma)$, which is freely generated over intersection points between two auxiliary Lagrangians in an auxiliary symplectic manifold. The differential, $\partial_{\CF}$, counts certain pseudo-holomorphic curves in the auxiliary symplectic manifold. The knot Floer homology of a $2n$ pointed knot $K$ in a three manifold $M$ is exactly the sutured Floer homology of $(Y(K)),\gamma(K,n))$.

\section{Topology and Heegaard diagrams}\label{sec:Heegaarddiagrams}
In this section we discuss Heegaard diagrams for knots and links obtained from sutured Heegaard diagrams for their Seifert surface complements. In Section~\ref{subsec:knotcomplementstosuturedmanifolds} we discuss how to glue certain sutured manifolds together to get sutured Seifert surface and knot exteriors. In Sections~\ref{subsec:HDsurfaceexterior}, \ref{subsec:HDformanifolds}, and \ref{subsec:HFforknots} we discuss how to obtain Heegaard diagrams for various manifolds and knots. In Section~\ref{subsec:admissibility} we verify that these Heegaard diagrams are admissible, i.e. can be used to compute Heegaard Floer homology. Finally, in Section~\ref{subsec:gradings}, we remind the reader how to determine the Alexander gradings of generators of the knot Floer chain complex from such Heegaard diagrams.

\subsection{From sutured manifolds to sutured knot exteriors}\label{subsec:knotcomplementstosuturedmanifolds}
Let $(G,\gamma)$ be an arbitrary sutured manifold. 
Let $(P,\rho)$ be a product sutured manifold with $|\rho|=|\gamma|+1$.
Let $(Y,\gamma_0)$ be a sutured manifold obtained by identifying $|\gamma|$ of $(P,\rho)$'s sutures with those of $(G,\gamma)$.

Pick a self-indexing Morse function $f_P$ on $P$ without critical points, the gradient flow of $f_P$ gives a diffeomorphism $\mathfrak{i}_P:R_+(\rho)\to R_-(\rho)$. This diffeomorphism can be extended to a diffeomorphism ${\mathfrak{i}:R_+(\gamma_0)\to R_-(\gamma_0)}$. Pick and fix such a diffeomorphism $\mathfrak{i}$ for the rest of this paper.

We will consider knots admitting Seifert surfaces that are diffeomorphic to $(Y,\gamma_0)$. That is, we will consider knots whose exteriors can be reconstructed as
\begin{align*}Y/(\mathfrak{i} \circ \psi):R_+(\gamma_0)\to R_-(\gamma_0)\end{align*} 
for a choice of diffeomorphism (``monodromy'') $\psi:R_+(\gamma_0)\to R_+(\gamma_0)$ such that $ {\psi|_{\partial R_+(\gamma_0)}=id|_{\partial R_+(\gamma_0)}}$. Here we equip the single torus boundary component of $Y/(\mathfrak{i} \circ \psi)$ with a pair of meridians given by the images in $Y/(\mathfrak{i} \circ \psi:R_+(\gamma_0)\to R_-(\gamma_0))$ of the pair of flow lines of $f_{P}$ in $(Y,\gamma_0)$ containing some pair of points in $\partial R_+(\gamma_0)$. We denote the resulting knot by $K_\psi$; strictly speaking, it also depends on the choice of $i$, which we suppress from the notation.

\subsection{Heegaard Diagrams for sutured manifolds.}\label{subsec:HDsurfaceexterior} Suppose $(G,\gamma)$ and $(P,\rho)$ are as in the previous subsection. $(P,\rho)$ admits a Heegaard diagram with underlying surface, $\Sigma_P$, the base surface for $(P,\rho)$ and no $\alpha$ or $\beta$-curves. We call $\Sigma_P$ the \emph{middle piece}. Note that $\Sigma_P$ has $|\gamma|+1$ boundary components. $(G,\gamma)$ admits some Heegaard diagram $(\Sigma_G,\bm{\alpha_G},\bm{\beta_G})$. We call $\Sigma_G$ the \emph{top piece}.  In particular, $\Sigma_G$ has $|\gamma|$ boundary components. Set ${\Sigma_{G\cup P}:=\Sigma_G\cup\Sigma_P}$. A sutured Heegaard diagram $\mathcal{H}(-S):=(\Sigma_{G\cup P},\bm{\beta_G},\bm{\alpha_G})$ for $(-Y(S),-\gamma(S))$ can be constructed by gluing $\Sigma_P$ to $\Sigma_G$ along $|\gamma|$ boundary components. Set $g:=g(R_\pm(\gamma_0))$; this will end up being the genus of a Seifert surface for the knots whose Heegaard diagrams we construct in Section~\ref{subsec:HFforknots}. Notice that $\Sigma_{G\cup P}$ has genus $g+k$. We also write $\bm{\alpha_G}=\{\alpha_i\}_{i=2g+1}^{2g+k}$ and $\bm{\beta_G}=\{\beta_i\}_{i=2g+1}^{2g+k}$.

\subsection{Heegaard diagrams for $3$-manifolds.}\label{subsec:HDformanifolds} We will be interested in knots in $3$-manifolds admitting Heegaard diagrams of a special form we shall now describe. 

Let $\mathcal{H}(-S)=(\Sigma_{G\cup P},\bm{\beta_G},\bm{\alpha_G})$ be a sutured Heegaard diagram as in Section~\ref{subsec:HDsurfaceexterior}.   Let $\Sigma_S$ be a surface with a single boundary component and genus $g$. We will call $\Sigma_S$ the \emph{bottom piece} of the Heegaard diagram. Set $\Sigma:=\Sigma_S \cup \Sigma_{G\cup P}$. 

Pick a set of arcs
$\bm{{\overline{a}}}:=\{\overline{{a_i}}\}_{i=1}^{2g}$ in $\Sigma_{S}$ that are pairwise disjoint and form a basis for $H_1(\Sigma_{S},\partial\Sigma_{S})$. Let $\bm{\overline{{b}}}:=\{\overline{{b_i}}\}_{i=1}^{2g}$ be arcs in $\Sigma_{S}$ where $b_i$ is a push-off of $a_i$ that, near the boundary, is in the direction induced by the orientation of $\Sigma_S$ for all $i$ on $\Sigma_S$. In particular $a_i$ and $b_i$ intersect transversely at one point for each $i$. See Figure~\ref{Fig: standard HD with multi basepoints} for an example.

Pick a set of arcs $\bm{a}:=\{a_i\}_{i=1}^{2g}$ in $\Sigma_{G\cup P}$ with ${\partial a_i=\partial \overline{a_i}}$ such that the arcs $a_i$ and $\alpha$-curves in $\bm{\alpha_G}$ are pairwise disjoint and together form a basis of $H_1(\Sigma_{G\cup P},\partial\Sigma_{G\cup P})$. 
Pick a self-indexing Morse function $f_G$ on $(G,\gamma)$. There is a self-indexing Morse function, $f_{G\cup P}$, on $(Y,\gamma_0)$ obtained by patching $f_G$ and $f_P$ together.
Let $h_-$ be the map defined on all but a finite number of points in $R_{-}(\gamma_0)$ that sends each point $x\in R_-(\gamma_0)$ to its images in $\Sigma_{G\cup P}$ under the gradient flow of $f_{G\cup P}$. Let $h_+$ be the map defined on all but a finite number of points in $\Sigma_{G\cup P}$ that sends each point $x\in \Sigma_{G\cup P}$ to its images in $R_+(\gamma_0)$ under the gradient flow of $f_{G\cup P}$. After a small perturbation of $\psi$ on the interior of $R_{\pm}(\gamma_0)$ if necessary, set $\bm{b}:=\{{b_i}\}_{i=1}^{2g}$ in $\Sigma_{G\cup P}$ where ${b_i}=h_-\circ\psi\circ h_+({a_i})$. We further assume all the arcs in $\bm{a}$, $\bm{b}$,  $\bm{\alpha_G}$, and $\bm{\beta_G}$ are transverse to each other.

Now we can form a Heegaard diagram $\mathcal{H}(-M):=(\Sigma,\bm{\beta_S}\cup\bm{\beta_G},\bm{\alpha_S}\cup\bm{\alpha_G})$, where ${\bm{\alpha_S}=\{{\alpha_i}\}_{i=1}^{2g}}$ with ${\alpha_i=a_i\cup\overline{a_i}}$ for $1\leq i\leq 2g$ and 
$\bm{\beta_S}=\{{\beta_i}\}_{i=1}^{2g}$ with $\beta_i=b_i\cup\overline{b_i}$ for $1\leq i\leq 2g$. 
This Heegaard diagram describes some $3$-manifold $-M$ which depends on the choice of $\bm{a},\bm{b}, (G,\gamma), (P,\rho)$, and $\overline{\bm{a}}$.

Notice that, in $\Sigma_S$, each $\alpha_i$ curve intersects a unique $\beta$-curve --- namely $\beta_i$ --- and does so exactly once. In the case that $\Sigma_G=\emptyset$, the union of these intersection points would be Honda, Kazez and Mati\'c's reinterpretation of the contact class for an appropriate contact structure on $M$~\cite{honda2006contact}.

\subsection{Heegaard diagrams for knots.}\label{subsec:HFforknots}

Next we modify $\mathcal{H}(-M)$ to encode the knot, $-K_\psi$, which will be an isotopic copy of $-\partial\Sigma_S=-\partial\Sigma_{G\cup P}$ in $-M$. There are three ways in which we will do this.
\subsubsection{The first Heegaard diagram.}\label{subsubsec:H1knot}  $\mathcal{H}_1(-K_\psi)$ is obtained by adding $4g$ $z$ basepoints ${\bm{z}=\{z_1,z_2,...,z_{4g}\}}$, and $4g$ $w$ basepoints $\bm{w}=\{w_1,...,w_{4g}\}$ to $\mathcal{H}(-M)$, as well as $4g-1$ more $\alpha$-curves $\{\widehat{\alpha}_1,...,\widehat{\alpha}_{4g-1}\}$ and $4g-1$ more $\beta$-curves $\{\widehat{\beta}_1,...,\widehat{\beta}_{4g-1}\}$ to $\Sigma$ in a neighborhood of $\partial\Sigma_S$ in a manner which we shall now describe. We refer the reader to Figure~\ref{Fig: standard HD with multi basepoints} for a picture of a neighborhood of $\partial\Sigma_S$.

Consider $\partial \Sigma_S$. Travel around $\partial\Sigma_S$ in the direction opposite to that dictated by the orientation of $\Sigma_S$, starting at $z_1$ in the large region. Immediately upon passing the $i$th intersection between a $\beta$-curve and $\partial\Sigma_S$, place the $w_i$ basepoint. Immediately upon passing the $i$th $\alpha$-curve, place the $z_{i+1}$ basepoint.  Now add $\alpha$-curves $\widehat{\alpha}_i$ as the boundary of a small neighborhood of the sub-arc of $\partial\Sigma_S$ connecting $z_{i+1}$ to $w_{i+1}$ that goes against the direction dictated by the orientation of $\Sigma_S$ for $i\neq 1$. Also add $\beta$-curves $\widehat{\beta}_i$ as the boundary of a small neighborhood of the sub-arc of $\partial\Sigma_S$ connecting $w_i$ and $z_{i+1}$ that goes against the direction dictated by the orientation of $\Sigma_S$. Set $\widehat{\bm{\alpha}}:=\{\widehat{\alpha}_i\}_{i=1}^{4g-1}$ $\widehat{\bm{\beta}}:=\{\widehat{\beta}_i\}_{i=1}^{4g-1}$, $\bm{z}:=\{z_i\}_{i=1}^{4g}$, and $\bm{w}=\{w_i\}_{i=1}^{4g}$. We then set $\mathcal{H}_1(-K_\psi):=(\Sigma, \bm{\beta_S}\cup\bm{\beta_G}\cup \widehat{\bm{\beta}}, \bm{\alpha_S} \cup  \bm{\alpha_G} \cup \widehat{\bm{\alpha}},\bm{w},\bm{z})$.

\begin{figure}[htb!]
\centering
\begin{tikzpicture}
\begin{scope}[thin, black!0!white]
          \draw  (-5, 0) -- (5, 0);
      \end{scope}
        \node at (0,0){ \includegraphics[width=0.5\textwidth]{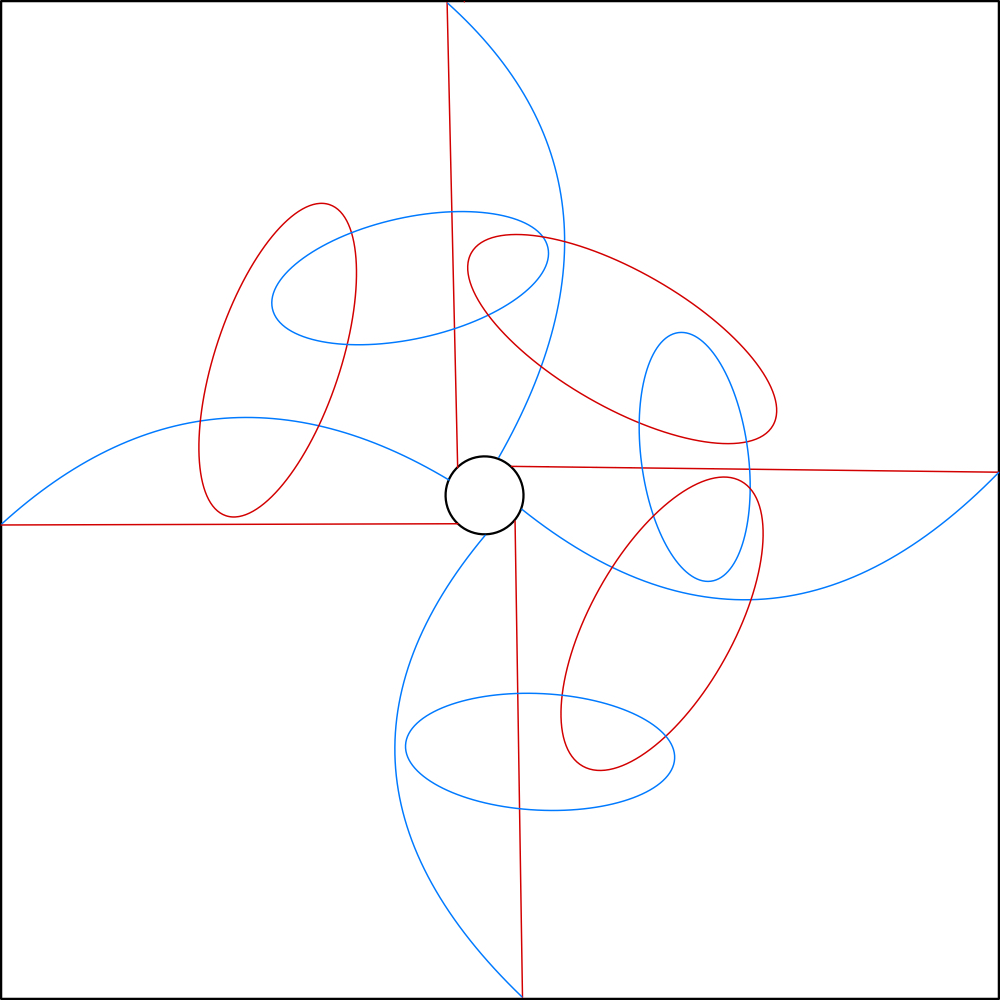}};
        \filldraw[black] (0.19,-4.1) circle (2pt);
        \node at (0.7,-3.9)  {$c_1$};
        \filldraw[black] (4.1,0.19) circle (2pt);
        \node at (3.9,0.7)  {$c_2$};
        \filldraw[black] (1.35,-1.95) circle (2pt);
        \node at (1.7,-2.4)  {$\widehat{c}_1$};
        \filldraw[black] (2.02,0.48) circle (2pt);
        \node at (2.5,0.45)  {$\widehat{c}_2$};
        \draw[dotted] (-0.2,-0.1) circle (2); 
        
        \filldraw[black] (-1.24,2.19) circle (2pt);
        \node at (-1.1,2.6)  {$\widehat{c}_3$};
        \node at (0.5,-3.3){\textcolor{red}{$\alpha_1$}};
        \node at (-1,-3.3){\textcolor{blue}{$\beta_1$}};
        \node at (-0.2,-2.1){{$w_1$}};
        \node at (-0.3,-1.3){\textcolor{blue}{$\widehat{\beta}_1$}};
        \node at (0.9,-2){\textcolor{black}{$z_2$}};
        \node at (2,-1.7){\textcolor{red}{$\widehat{\alpha}_1$}};
        \node at (3,-1){\textcolor{blue}{${\beta}_2$}};
        \node at (3,-0.2){\textcolor{red}{${\alpha}_2$}};
        \node at (1.5,0.9){\textcolor{black}{${z}_3$}};
        \node at (1.6,2){\textcolor{red}{$\widehat{\alpha}_2$}};
        \node at (0.05,1.9){\textcolor{black}{$w_3$}};
        \node at (-1.5,1.7){\textcolor{black}{$z_4$}}; 
         \node at (-0.9,0.9){\textcolor{blue}{$\widehat{\beta}_3$}};
         \node at (-2.6,1.3){\textcolor{red}{$\widehat{\alpha}_3$}};
         \node at (-2.1,0.4){$w_4$};
        \node at (1.7,-0.2){\textcolor{black}{$w_2$}};
        \node at (0.9,-0.1){\textcolor{blue}{$\widehat{\beta}_2$}};
         \node at (-2,-2){ $z_1$};

\end{tikzpicture}
    \caption{A neighborhood of $\Sigma_S$ in the Heegaard diagram $\mathcal{H}_1(K)$, for $K$ a genus one knot. The dotted circle is the boundary of $\Sigma_S$. Note that the right and left hand sides of the diagram are identified, as are the top and bottom sides. The $c_i$ and $\widehat{c}_i$ intersection points are the intersection points we consider in Section \ref{sec:minimalgrading}.} \label{Fig: standard HD with multi basepoints}
\end{figure}
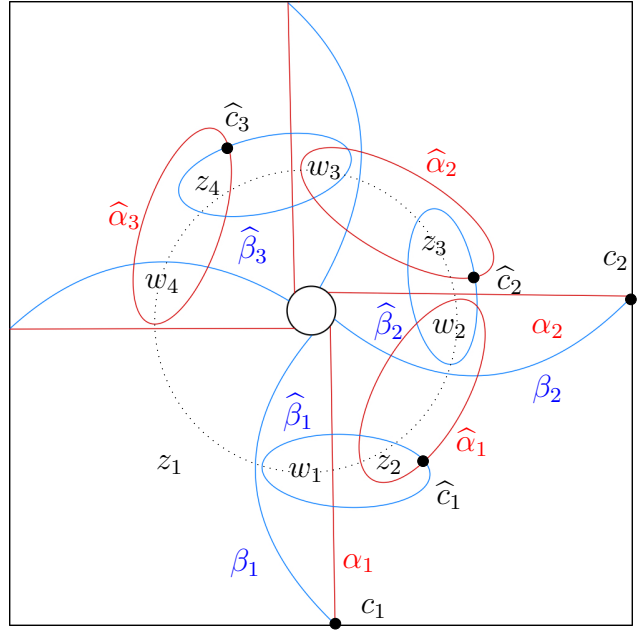

\subsubsection{The second Heegaard diagram.} $\mathcal{H}_2(-K_\psi)$, is obtained by modifying $\mathcal{H}_1(-K_\psi)$ as follows. Perform $2g(4g-1)$ handle slides of all of the $\bm{\beta_S}$-curves over $\widehat{\bm{\beta}}$-curves in a neighborhood of the arc from $w_{4g}$ to $z_2$ in the direction induced by the orientation of $\Sigma_S$, so that after performing this operation the only $\beta$-curves that intersect $\partial\Sigma_S$ in the arc from $z_1$ to $w_1$ in the direction induced by the orientation of $\Sigma_S$ are curves in $\widehat{\bm{\beta}}$. In a minor abuse of notation we will continue to let $\alpha_i,\widehat{\alpha}_i,\beta_i$, and $\widehat{\beta}_i$ denote the images of the curves $\alpha_i,\widehat{\alpha}_i,\beta_i,$ and $\widehat{\beta}_i$ under these handle-slides.

\subsubsection{The third Heegaard diagram.}\label{subsubsec:HFhat} $\widehat{\mathcal{H}}(-K)$ is the doubly pointed Heegaard diagram defined as follows. Start with $\mathcal{H}_2(-K_\psi)$ and then remove every basepoint other than $z_1$ and $w_{2g}$ along with every $\widehat{\alpha}$ and $\widehat{\beta}$-curve. Observe that in the special case that $(G,\gamma)=\emptyset$, this Heegaard diagram is exactly the type of Heegaard diagram for a fibered knot considered in~\cite[Section 2.4]{baldwin_note_2018}. Observe, as in the case for $\mathcal{H}(-M)$ discussed in Section~\ref{subsec:HDformanifolds}, that $\widehat{\mathcal{H}}(-K)$ is determined by the following data:\begin{itemize}
    \item $(\Sigma_G,\bm{\alpha_G},\bm{\beta_G})$, a Heegaard diagram for the $(G,\gamma)$, 
 \item A collection of properly embedded arcs $\bm{a}$ in $ \Sigma_{G\cup P}$, such that $\bm{a}\cup \bm{\alpha_G}$ are linearly independent in  $H_2(\Sigma_{G\cup P},\partial(\Sigma_{G\cup P}))$.
 \item    a collection of properly embedded arcs $\bm{b}$ in $\Sigma_{G\cup P}\setminus \bm{\beta_G}$ whose endpoints coincide with those of a push-off of $a_i$ such that $\bm{b}\cup\bm{\beta_G}$ is a basis for $H_2(\Sigma_{G\cup P},\partial\Sigma_S)$.
 
 \end{itemize}

 When we wish to emphasize the dependence of $\widehat{\mathcal{H}}(-K)$ on this data we write $${\widehat{\mathcal{H}}(-K)=\widehat{\mathcal{H}}(\Sigma_G,\bm{\beta_G},\bm{\alpha_G},\bm{{b}},\bm{{a}},\overline{\bm{a}}}). $$ 
 
 Note that the data $\bm{{b}}$ is equivalent to the data $\mathfrak{i}\circ\psi$. Given a map $R_+(\gamma)\to R_-(\gamma)$, one can readily construct $\bm{b}$ from $\bm{a}$ by considering the flow of the image of the flow of the arcs $\bm{a}$. Conversely, given $\bm{b}$ and $\bm{a}$, (and any $\mathfrak{i}:R_+(\gamma)\to R_-(\gamma)$) one can construct a diffeomorphism $\psi:R_+(\gamma)\to R_+(\gamma)$ by defining it so that $\mathfrak{i}\circ\psi$ sends the images of $\bm{a}$ under the gradient flow, to curves which flow to $\bm{b}$. $\psi$ can be extended over $R_+(\gamma)$ using Alexander's theorem.

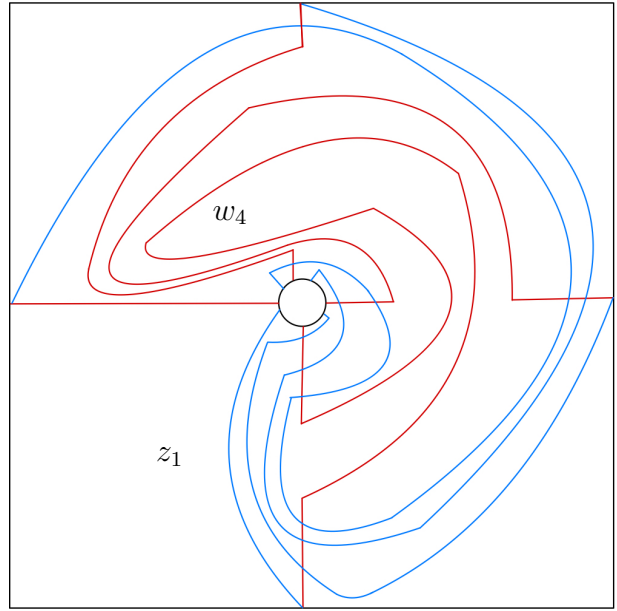
\begin{figure}[htb!]
\centering
\begin{tikzpicture}
\begin{scope}[thin, black!0!white]
          \draw  (-5, 0) -- (5, 0);
      \end{scope}
        \node at (0,0){ \includegraphics[width=0.5\textwidth]{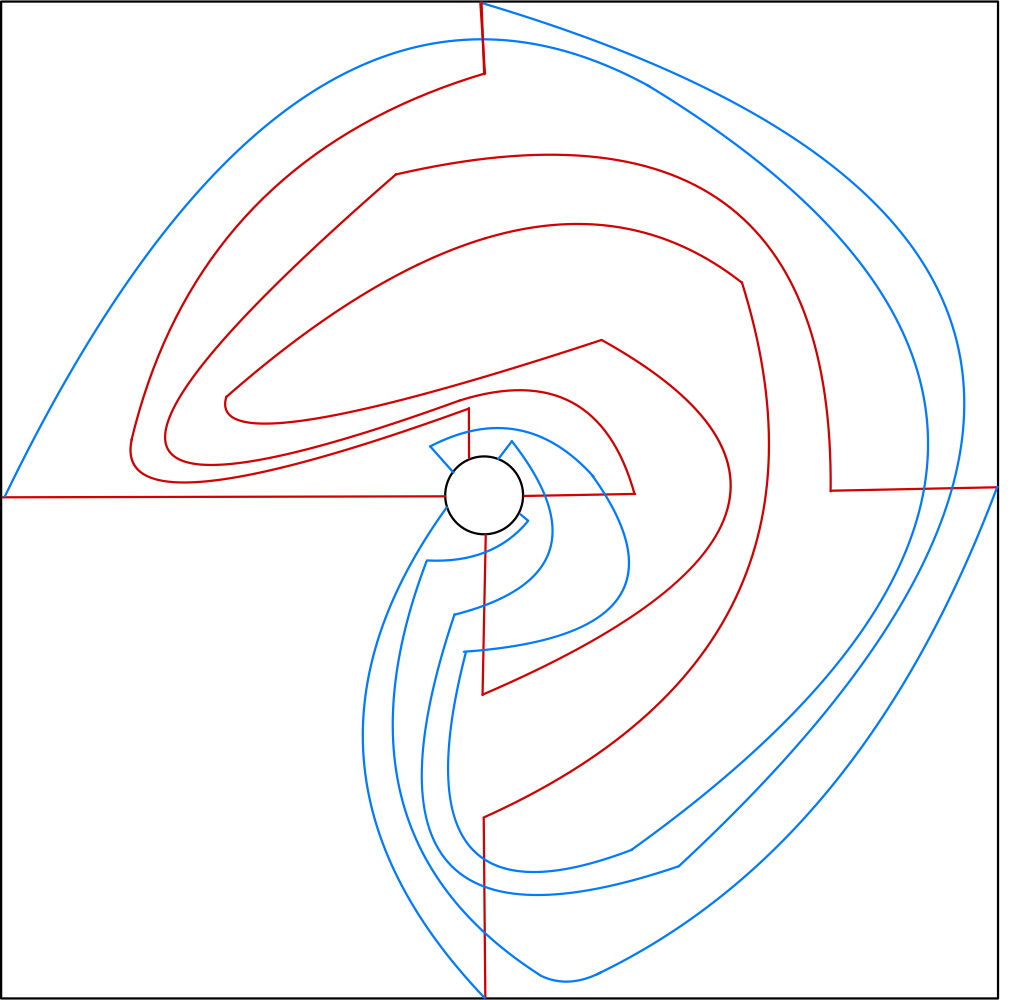}};
         \node at (-2,-2){ $z_1$};
          \node at (-1.2,1.2){ $w_4$};
\end{tikzpicture}
     \caption{The $\Sigma_S$ portion of $\widehat{\mathcal{H}}(-K)$ in the case that $K$ is of genus one. The left and right hand sides of the figure should be identified, as should the top and bottom.} \label{Fig: standard HD with two basepoints}
\end{figure}

\subsubsection{Existence}The following result is well known;
\begin{lemma}
    Any null-homologous knot $K$ in any $3$-manifold admits a Heegaard diagram as constructed above with $g(\Sigma_S)$ equal to the Seifert genus of $K$.
\end{lemma}

\begin{proof} Such a diagram can be constructed using, say, a circle valued Morse function as in~\cite{lekili2013heegaard}.
\end{proof}

\subsection{Admissibility}\label{subsec:admissibility}

We need to ensure that the Heegaard diagrams $\widehat{\mathcal{H}}(-K_\psi)$, $\mathcal{H}_1(-K_\psi)$, and $\mathcal{H}_2(-K_\psi)$ can be made admissible.

\begin{lemma}
    If $(\Sigma_G,\bm{\alpha_G},\bm{\beta_G})$ is admissible then so too are $\widehat{\mathcal{H}}(-K_\psi)$, $\mathcal{H}_1(-K_\psi)$, and $\mathcal{H}_2(-K_\psi)$.
\end{lemma}

\begin{proof}
    This follows immediately from the positioning of the basepoints in $\Sigma_S$ near the $c_i$ and $\widehat{c}_i$ intersection points.
\end{proof}

Since $(G,\gamma)$ always admits an admissible Heegaard diagram it follows that we may always assume that the Heegaard diagrams $\widehat{\mathcal{H}}(-K_\psi)$, $\mathcal{H}_1(-K_\psi)$, and $\mathcal{H}_2(-K_\psi)$ are admissible. We duly make this assumption henceforth.
\subsection{Alexander Gradings}\label{subsec:gradings}
Let $A_\text{min}$ be the minimum Alexander grading in which $${\widetilde{{\HFK}}(\mathcal{H}_1(-K_\psi)\cong\widetilde{{\HFK}}(-Y,-K_\psi,[-S]))}$$ is non-trivial. It follows, say, from work of~\cite{juhasz2008floer}, that for $K$ non-trivial, $A_\text{min}=\frac{1}{2}-3g(K)$. For the Heegaard diagrams discussed in Section~\ref{subsec:HFforknots}, it is easy to read off the Alexander grading of any generator $\bm{x}\in\T_{\bm{\alpha}}\cap\T_{\bm{\beta}}$ relative to $[-S]$ via the following formula:
\begin{equation}\label{eq:Alexandergrading}
    A_{[-S]}=A_\text{min}+\frac{1}{2}\cdot|\{x\in\bm{x}:x\not\in\Sigma_S\}|.
\end{equation}

\noindent This follows, for example, from an argument given in the proof of~\cite[Lemma 2.1]{MR3071144}.

\section{Generators of minimal Alexander grading}\label{sec:minimalgrading}

In this section we describe certain elements in the knot Floer chain complex computed from the diagrams we constructed in the previous subsection. The choices of elements we consider are inspired by the LOSS invariant~\cite{lisca2009heegaard}, the braid invariant~\cite{baldwin2013equivalence}, and Honda, Kazez, and Mati\'c's reinterpretations~\cite{honda2006contact} of the Ozsv\'ath and Szab\'o's contact invariant~\cite{ozsvath2005heegaard}. We continue to work in the setting of Section~\ref{sec:Heegaarddiagrams}. The diagram we ultimately want to consider is  $\widehat{\mathcal{H}}(-K_\psi)$, which is compatible with $(-M,-K_\psi)$. 

$\mathcal{H}(-M)$ --- the Heegaard diagram constructed in Section~\ref{subsec:HDformanifolds} --- contains unique intersection points $\alpha_i\cap\beta_i\subset \Sigma_S$, which we denote by $c_i$. These intersection points persist under the operations used to obtain $\mathcal{H}_1(-K_\psi)$, $\mathcal{H}_2(-K_\psi)$, and $\widehat{\mathcal{H}}(-K)$ where, by a minor abuse of notation, we still denote them by $c_i$. Note, moreover, that in $\mathcal{H}_1(-K_\psi)$, $\widehat{\alpha}_i \cap \widehat{{\beta}}_i\cap\Sigma_S$ consists of a single intersection point  for $1\leq i \leq 4g-1$, which we denote by $\widehat{c}_i$. Abusing notation again, set $\bm{c}$ to be $\{c_i\}_{1\leq i\leq 2g}$ for $\widehat{\mathcal{H}}(-K_\psi)$, and $\{c_i\}_{1\leq i\leq 2g}\cup\{\widehat{c}_i\}_{1\leq i\leq 4g-1}$ for $\mathcal{H}_1(-K_\psi)$.

We define functions $$\widetilde{f}_\psi:\CF(\mathcal{H}(-S))\to \widetilde{\CFK}(\mathcal{H}_1(-K_\psi))$$ and $$f_\psi:\CF(\mathcal{H}(-S))\to \widehat{\CFK}(\widehat{\mathcal{H}}(-K_\psi)).$$  Each map is defined on basis elements of $\mathcal{H}(-S)$, that is elements of $\T_{\bm{\alpha_G}}\cap\T_{\bm{\beta_G}}$, via $\bm{e}\mapsto \bm{e}\cup\bm{c}$ and extending linearly. Here we are again abusing notation by allowing $\bm{e}$ to denote both a collection of intersection points in  $\mathcal{H}(-S)$ and a collection of intersection points in either $\mathcal{H}_1(-K_\psi)$ or $\widehat{\mathcal{H}}(-K_\psi)$. It follows from Equation~\ref{eq:Alexandergrading}, that these maps can be viewed as maps $$\widetilde{f}_\psi:\CF(\mathcal{H}(-S))\to\widetilde{\CFK}(\mathcal{H}_1(-K_\psi),A_{\text{min}}),$$

and $${f}_\psi:\CF(\mathcal{H}(-S))\to\widehat{\CFK}(\widehat{\mathcal{H}}(-K_\psi),A_{\text{min}}).$$

\begin{definition} \label{def: H(sigma,alphaG,betaG,a)}

    We say that $\bm{x}\in \widehat{\CFK}(\widehat{\mathcal{H}}(-K_\psi),A_{\text{min}})$ (or $\widetilde{\CFK}(\mathcal{H}_1(-K_\psi),A_{\text{min}})$) is a \emph{contact class like element} if it is in the image of $f_\psi$ (or $\widetilde{f}_\psi$).
\end{definition}

 This definition is dependent on the data of $(\Sigma_G,\bm{\beta_G},\bm{\alpha_G},\bm{b},\bm{\alpha_S})$ which we have seen in section \ref{subsubsec:HFhat}. We will suppress this dependence when these choices are apparent from the context. Indeed, in Section~\ref{subsec:independence} we show that being generated by contact class like elements is in fact independent of all of this auxiliary data.
 
 The primary goal of this section is to show that $f_\psi$ and $\widetilde{f}_\psi$ induce isomorphisms on homology when $K_\psi$ has $w$-avoiding exterior, or, equivalently, that $\widehat{\HFK}(\widehat{\mathcal{H}}(-K_\psi),A_\text{min})$ is generated by contact class like generators.

\begin{lemma}\label{lem:chain}
    The maps $$\tilde{f}_{\psi}:\CF(\mathcal{H}(-S))\to\widetilde{\CFK}(\mathcal{H}_1(-K_{\psi}),A_\text{min})$$ and $${{f}_{\psi}:\CF(\mathcal{H}(-S))\to\widehat{\CFK}(\widehat{\mathcal{H}}(-K_{\psi}),A_\text{min})}$$ are chain maps for any diffeomorphism $\psi$.
\end{lemma}

Note that Lemma~\ref{lem:chain} holds without the hypothesis that $K_\psi$ has $w$-avoiding exterior.

\begin{proof}
   We prove the $\tilde{f}_{\psi}$ case. The proof in the $f_{\psi}$ case is essentially identical. It suffices to show that $\partial_{\widetilde{\CFK}}\tilde{f}_\psi(\bm{x})=\tilde{f}_\psi(\partial_{\CF}\bm{x})$ for all generators $\bm{x}\in\T_{\bm{\alpha_G}}\cap\T_{\bm{\beta_G}}$. Observe that the positioning of the basepoints near any element of $\bm{c}$ --- as in Figure~\ref{fig:standardtraingles} (ignoring the green arc and orange region) --- any pseudo-holomorphic disk from $\tilde{f}_\psi(\bm{x})$ must be constant on $c_i$ for all $i$. It follows that there is an exact correspondence between disks which contribute to the differential $\partial_{\CF}\bm{x}$ and $\partial_{\widetilde{\CFK}}\widetilde{f}_\psi(\bm{x})$. The result follows.
\end{proof}

Set $\tau$ to be a positive Dehn twist about the boundary on $R_+(\gamma)$. Recall that any diffeomorphism $\psi:R_+(\gamma_0)\to R_+(\gamma_0)$ can be obtained from an appropriate $ \tau^n$ by post-composing with a sequence of negative Dehn twists about some collection of simple closed curves in $R_+(\gamma_0)$. In Subsection~\ref{subsec:boundarytwists} we prove that $\tilde{f}_{\tau^n}$ and $f_{\tau^n}$ induces isomorphisms between $\SFH(\mathcal{H}(-S),\gamma(S))$ and $\widehat{\HFK}(\widehat{\mathcal{H}}(-K_\psi),A_\text{min})$ for all $n$. In Subsection~\ref{subsec:arbitrarydiffeos} we reduce the general case to the boundary twisting case.

\subsection{Boundary twists}\label{subsec:boundarytwists}
In this section we prove that $f_{\tau^n}$ and $\widetilde{f}_{\tau^n}$ induce isomorphisms on homology. We assume throughout this section that the relevant knots have $w$-avoiding exteriors.

\begin{proposition}\label{prop:hatisiso}
    For any $n$, the map ${f}_{\tau^n}^*:\SFH(\mathcal{H}(-S))\to \widehat{\HFK}(\widehat{\mathcal{H}}(-K_{\tau^n}),A_\text{min})$ is an isomorphism. Equivalently, $\widehat{\HFK}(\widehat{\mathcal{H}}(-K_{\tau^n}),A_\text{min})$, is generated by contact class like elements.
  
\end{proposition}

 \begin{figure}[h]
 \centering
\begin{tikzpicture}

\fill[orange!30] (1,0) -- (3,0) -- (3,-2) -- (1,0);

\draw[dashed] (0,-3) rectangle (4,1);

\draw[blue, thick] (0,0) -- (4,0);

\draw[red, thick] (3,1) -- (3,-3);

\draw[green, thick] (0,1) -- (4,-3);

     \filldraw[black] (1,0) circle (2pt);
    \node at (1,-0.3) {\( \theta_i \)};

    \filldraw[black] (3,0) circle (2pt);
    \node at (3.2,-0.2) {\( c_i \)};

      \filldraw[black] (3,-2) circle (2pt);
    \node at (2.8,-2) {\( c_i' \)};


\node at (0.25,0.25) {\( D_1\)};

\node at (2,0.25) {\( D_2\)};

\node at (3.25,0.25) {\( 0\)};

\node at (0.5,-1) {\( 0\)};

\node at (2.5,-1) {\( D_3\)};

\node at (3.25,-1) {\( D_4\)};

\node at (3.25,-2.7) {\( D_5\)};

\end{tikzpicture}

    \caption{We will frequently use the fact that near elements $c_i$ of $\bm{c}$, the shadow of any pseudo-holomorphic triangle is as shown.}\label{fig:standardtraingles}
   \end{figure}
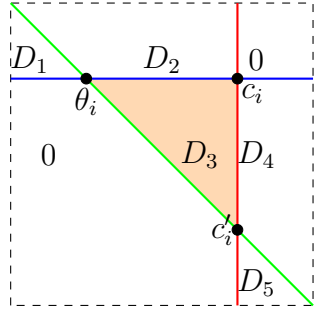

 We will prove this proposition by way of a sequence of Lemmas. For the rest of this subsection we assume that our Heegaard diagrams satisfy some additional conditions:
 
 \begin{enumerate}
\item  Exactly one $\bm{a}$ arc, $a_G$,  intersects $\Sigma_G$ non-trivially.
\item  There is exactly one arc $a_s$ in $\bm{a}$ which is separating in $\Sigma_P$, separating it into two pieces $\Sigma_{P,1}\cup \Sigma_{P,2}$.

\item The only arc contained in $\Sigma_{P,2}$, is one of two components of $a_G\cap \Sigma_P$, and that $z_1,w_1$ are contained in the boundary of $\Sigma_{P,1}$. See Figure~\ref{fig:protoHD}.

\end{enumerate}

 Let $\alpha_G$ and $\alpha_s$ denote the $\alpha$-curves corresponding to the arcs $a_G$ and $a_s$ respectively. We can make these assumptions since we are assuming that the exterior of our chosen Seifert surface for $K$ has exterior obtained by gluing a product sutured manifold to a knot exterior along a pair of sutures; a consequence of the definition of knots with $w$-avoiding exteriors. With this assumption in place we can proceed to prove the following lemma.

\begin{lemma}\label{lem:inj}
   The map $\tilde{f}_{\tau^n}^*:\SFH(\mathcal{H}(-S))\to\widetilde{\HFK}(\mathcal{H}_1(-K_{\tau^n}),A_\text{min})$ is injective.
\end{lemma}

   \begin{figure}[ht]
    \centering

   \begin{tikzpicture}


\draw[ thick]
(3,3) -- (3,1)
.. controls (3,0) and (3,-3) ..
(1,-3) -- (-2,-3)
.. controls (-4,-3) and (-4,0) ..
(-4,1) --  (-4,3);

\draw[ thick](1,3)
.. controls (1,2) and (-1,2) ..
(-1,3) ;

\draw[pink, thick] (2,3) ellipse (1 and 0.5);

\draw[pink, thick] (-2.5,3) ellipse (1.5 and 0.5);

\draw[orange, thick] (-4,0)
.. controls (-2,-0.5) and (2,-0.5) .. (3,0);
\draw[orange, thick, dashed] (-4,0)
.. controls (-2,0.5) and (2,0.5) .. (3,0);

\draw[orange, thick] (-4,0.5)
.. controls (-2,0) and (2,0) .. (3,0.5);
\draw[orange, thick, dashed] (-4,0.5)
.. controls (-2,1) and (2,1) .. (3,0.5);

\draw[draw=black, thick] (0,-2) ellipse (0.5 and 0.3);
\draw[draw=green, thick] (0,-2) ellipse (0.6 and 0.4);

\draw[draw=black, thick] (-2.5,-1.5) ellipse (0.3 and 0.2);

\draw[draw=black, thick] (-2.5,1.5) ellipse (0.3 and 0.2);


\draw[red, thick] (0,-1.7)
.. controls (-0.3,-1) and (-0.3,0) .. (0,2.25);
\draw[red, thick, dashed] (0,-1.7)
.. controls (0.3,-1) and (0.3,0) .. (0,2.25);

\draw[blue, thick] (-0.4,0.1)
.. controls (-0.4,1) and (0.2,2) .. (0.2,2.25);
\draw[blue, thick,dashed] (0.2,2.25) .. controls (0.2,2) and (0.3,2) .. (0.3,0.8);
\draw[blue, thick,dashed] (0.3,0.4) .. controls (0.3,-1.5) and (0.1,-1.7) .. (0.1,-1.7);
\draw[blue, thick] (-0.4,-0.4)
.. controls (-0.4,-1) and (0,-1.5) .. (0.1,-1.7);

\draw[red, thick] (-1.5,2.65)
.. controls (-1.5,-5) and (2,-4) .. (2,2.5);

\draw[blue, thick] (-1.65,2.6)
.. controls (-1.65,2) and (-1.7,1) .. (-1.6,0.2);

\draw[blue, thick] (-1.6,-0.3)
.. controls (-1.6,-3.5) and (1.3,-3.3) .. (1.6,-0.3);

\draw[blue, thick] (1.6,0.2)
.. controls (1.6,1) and (1.8,2) .. (1.8,2.5);

\draw[red, thick] (-2.5,-1.3)
.. controls (-2.8,-0.5) and (-2.8,0.5) .. (-2.5,1.3);
\draw[red, thick, dashed] (-2.5,-1.3)
.. controls (-2.2,-0.5) and (-2.2,0.5) .. (-2.5,1.3);

\draw[blue, thick] (-2.3,-1.35)
.. controls (-2.3,-0.5) and (-2.9,-1) .. (-2.9,-0.2);

\draw[blue, thick, dashed] (-2.3,-1.35)
.. controls (-2.3,-1) and (-2.1,-0.5) .. (-2.1,0.35);
\draw[blue, thick] (-2.9,0.3)
.. controls (-2.9,1) and (-2.35,1) .. (-2.35,1.3);

\draw[blue, dashed, thick] (-2.35,1.3)
.. controls (-2.35,1.1) and (-2.1,1) .. (-2.1,0.8);

\draw[red, thick] (-2.5,2)
.. controls (-4,2) and (-4,-2) .. (-2.5,-2) .. controls (-1.5,-2) and  (-1.5,2).. (-2.5,2);

\draw[blue, thick] (-3.7,0.4)
.. controls (-3.7,2) and (-2,3.5) .. (-2,0.2);

\draw[blue, thick] (-3.7,-0.05)
.. controls (-3.7,-2) and (-2,-3.5) .. (-2,-0.3);


        \node at (2,-2) {\(\Sigma_S \)};

        \node at (2.5,2) {\(\Sigma_{P,2} \)};

         \node at (-1,2) {\(\Sigma_{P,1} \)};

\end{tikzpicture}

    \caption{The pink curves indicate the boundary of $\Sigma_G$ --- the part of the Heegaard diagram corresponding to the top piece --- while the orange curves indicate the boundary of the twisting region. $a_S$ is the separating arc in $\Sigma_P$, $a_G$ is the unique arc which enters the top piece. $\Sigma_{P,1}$ is the component of $\Sigma_P\setminus a_S$ on the left hand side of the figure, while $\Sigma_{P,2}$ is the component on the right. The green curve will play a role in the proof of Theorem~\ref{prop:generalrankbound}.}\label{fig:protoHD}
\end{figure}
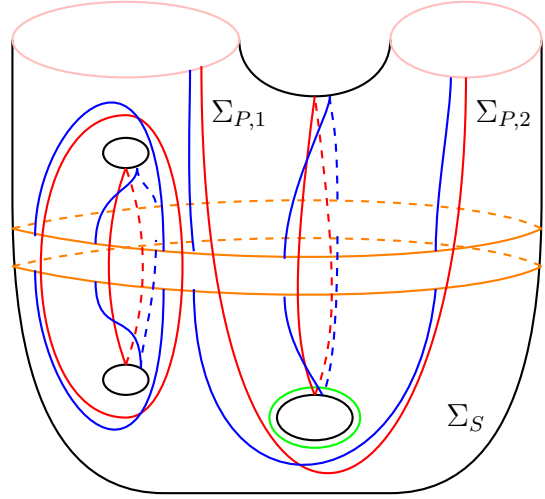  
  \begin{proof}
 Fix a Heegaard diagram $\mathcal{H}_1(-K_{\tau^n})$ as in Section~\ref{subsubsec:H1knot} satisfying the additional assumptions described above. We need to check that if there is a cycle $\bm{y}\in\CF(\mathcal{H}(-S))$  such that $\widetilde{f}_{\tau^n}(\bm{y})=\partial_{\widetilde{\CFK}} \bm{x}$ for some $\bm{x}\in\widetilde{\CFK}(\mathcal{H}_1(-K_{\tau^n}))$ then there is an $\bm{x}'\in\CF(\mathcal{H}(-S))$ with $\bm{y}=\partial_{\CF} \bm{x}'$. 
 It suffices to show that $\bm{x}=\widetilde{f}_{\tau^n}(\bm{z}')$ for some $\bm{z}'\in \CF(Y,\gamma_0)$. To see that this is sufficient, observe that the components of the differential from $\bm{x}$ are identified with the components of the differential from $\bm{z}'$ by the positioning of the basepoints near the $c_i$ and $\widehat{c}_i$, so that if $\bm{x}=\widetilde{f}_{\tau^n}(\bm{z}')$  then $\partial_{\CF}(\bm{z}')=\bm{y}$.

 Observe that the Heegaard diagram $\mathcal{H}_1(-K_{\tau^n})$ is standard outside of a neighborhood of a curve $\eta$ defined as a push-off of $\partial\Sigma_S$ into $\Sigma'$. We will refer to $\nu(\eta)$ as \emph{the twisting region}. See Figure~\ref{fig:protoHD}. Observe also that every intersection point $x$ in an $\alpha_i,\beta_i,\widehat{\alpha}_i$ or $\widehat{\beta}_i$ curve in an element of $\T_{\bm{\alpha}}\cap\T_{\bm{\beta}}$ with non-trivial coefficient in $\bm{x}$ must occur in $\Sigma_S$ so as to yield a generator of the correct Alexander grading in $\widetilde{\CFK}(\mathcal{H}_1(-K_{\tau^n}))$ by Equation~\ref{eq:Alexandergrading}.
 
 Let $\phi$ be a homotopy-class of pseudo-holomorphic disks from $\bm{x}$ to $\bm{y}$. Let $D(\phi)$ denote the domain of $\phi$ i.e. the projection of a lift of $\phi$ to $\Sigma^{\times (6g-1)}$ to $\Sigma$. Observe that in moving from one boundary component of the twisting region to the other boundary component in the complement of the $\alpha$-curves, one passes each $\beta$-curve zero times algebraically. Thus the multiplicity of a domain directly above the twisting region equals the multiplicity of the corresponding domain directly below. In particular, it follows that the large domain in $\Sigma_{P,1}$ --- where the label $\Sigma_{P,1}$ lies in Figure~\ref{fig:protoHD} ---  has multiplicity zero, since this is separated from the basepoint $z_1$ by the twisting region.  It follows in turn that a component of $\partial D(\phi)$ cannot cross into $\Sigma_{G\cup P}\setminus\nu(\partial\Sigma_S)$ at all save for on $a_s$, as all other such arcs have regions multiplicity of zero on both sides as they traverse $\Sigma_P$. Consequently, the multiplicities in domains separated by arcs $a_i\cap \Sigma_{G\cup S}$ or $\overline{a}_i\cap\Sigma_{G\cup S}$ are equal, where $a_i\neq a_s$.
 
 By examining the local multiplicities of $D$ in a a small neighborhood of $\widehat{\alpha}_1\cup\widehat{\beta_2}$, it again follows from the fact that the multiplicity of domains separated by a vertical arc in the twisting region agree that $\phi$ must be constant on  $\widehat{{\beta}}_1$ and then in turn   $\widehat{{\alpha}}_1$. Indeed, by proceeding round $\partial S$ against the orientation induced by $S$, we see that $\phi$ is constant on\begin{enumerate}
     \item  All curves $\widehat{{\beta}}_i$ with $i\leq 4g-3$; i.e. aside from $\widehat{\beta}_{4g-2}$,$\widehat{\beta}_{4g-1}$.
     \item All curves $\widehat{\alpha}_i$ with $i\leq 4g-3$; i.e. aside from $\widehat{\alpha}_{4g-2}$,$\widehat{\alpha}_{4g-1}$. 
     \item All $\alpha_i$ curves in $\bm{\alpha_S}$ other than $\alpha_s$ and $\alpha_{G}$.
     \item All $\beta_i$ curves in $\bm{\beta_S}$ other than $\beta_s$ and $\beta_{G}$.
 \end{enumerate}

  \noindent This allows us to ignore all $\bm{\alpha}$ and $\bm{\beta}$-curves other than those in the set $$\bm{\alpha_G}\cup\bm{\beta_G}\cup \{\widehat{\beta}_{4g-2},\widehat{\beta}_{4g-1},\widehat{\alpha}_{4g-2},\widehat{\alpha}_{4g-1}, \alpha_s,\alpha_G,\beta_s,\beta_G\}.$$

  Here, $\beta_G$ and $\alpha_G$ are the curves corresponding to the arc $a_G$, i.e the one which enters $\Sigma_G$. Likewise, $\beta_s$ and $\alpha_s$ are the curves corresponding to the arc $a_s$. 

 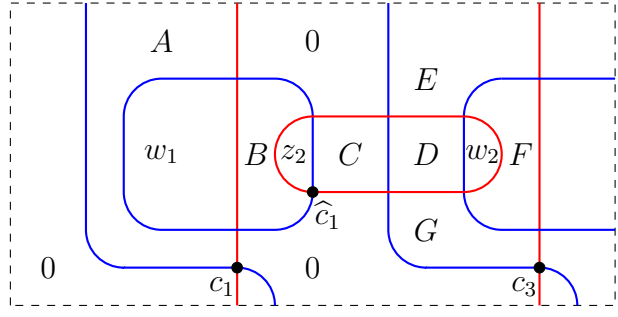
\begin{figure}
 
       \centering
\begin{tikzpicture}

\draw[dashed] (0,0) rectangle (8,4);

\draw[red, thick] (3,4)--(3,0);

\draw[red, thick] (7,4)--(7,0);

\draw[blue, thick] (1,4)--(1,1) arc[start angle=180, end angle=270, radius=0.5]--(1.5,0.5)--(3,0.5) arc[start angle=90, end angle=0, radius=0.5];

\draw[blue, thick] (5,4)--(5,1)arc[start angle=180, end angle=270, radius=0.5]--(5.5,0.5)--(7,0.5) arc[start angle=90, end angle=0, radius=0.5];

\draw[blue, thick] (2,3)--(3.5,3)arc[start angle=90, end angle=0, radius=0.5]-- (4,1.5)arc[start angle=0, end angle=-90, radius=0.5]--(2,1)arc[start angle=270, end angle=180, radius=0.5]--(1.5,2.5)arc[start angle=180, end angle=90, radius=0.5];

\draw[blue, thick] (6.5,3)--(8,3);
\draw[blue, thick](8,1) -- (6.5,1)arc[start angle=270, end angle=180, radius=0.5]--(6,2.5)arc[start angle=180, end angle=90, radius=0.5];

\draw[red, thick] (4,2.5)--(6,2.5)arc[start angle=90, end angle=-90, radius=0.5] --(4,1.5)arc[start angle=270, end angle=90, radius=0.5];


 \filldraw[black] (3,0.5) circle (2pt);
    \node at (2.8,0.25) {\( c_1 \)};

     \filldraw[black] (4,1.5) circle (2pt);
    \node at (4.2,1.2) {\( \widehat{c}_1 \)};

      \filldraw[black] (7,0.5) circle (2pt);
    \node at (6.8,0.25) {\({c}_3 \)};

      \node at (2,2) {\({w}_1 \)};

         \node at (3.75,2) {\({z}_2 \)};

          \node at (6.25,2) {\({w}_2 \)};

           \node at (0.5,0.5) {\(0 \)};

           \node at (4,0.5) {\(0 \)};

             \node at (4,3.5) {\(0 \)};

     \node at (2,3.5) {\(A \)};

        \node at (3.25,2) {\(B\)};

          \node at (4.5,2) {\(C\)};
            \node at (5.5,2) {\(D\)};
              \node at (5.5,3) {\(E\)};
              \node at (5.5,1) {\(G\)};

                  \node at (6.75,2) {\(F \)};

\end{tikzpicture}
\caption{A small rectangle containing $\widehat{\beta}_1\cup\widehat{\alpha}_1$. Similar pictures can be obtained for $\widehat{\beta}_i\cup\widehat{\alpha}_i$ for $i< 4g-3$.}
     \label{fig:movingalongboundary}
 \end{figure}

 With this in mind we consider a portion of a neighborhood of $\Sigma\cup \Sigma_{P,1}$ as shown in Figure~\ref{fig:}. We have neglected to include the curve $\widehat{\alpha}_{4g-3}$ since we know from above that the $\phi$ is constant on $\widehat{c}_{4g-3}$.

 \begin{figure}
 
       \centering
\begin{tikzpicture}

\draw[thick, dashed] (0,0) -- (12,0) -- (12,6)arc[start angle=0, end angle=180, radius =6] --(0,0);


\draw[thick, dashed] (5,8) -- (7,8) arc[start angle=90, end angle=-90, radius =0.5] --(5,7)arc[start angle=270, end angle=90, radius =0.5];


\draw[red, thick] (3,0)--(3,6)arc[start angle=180, end angle=0, radius=4]--(11,6)--(11,0);

\draw[orange, thick, dashed] (0,4)--(12,4);
\draw[orange, thick, dashed] (0,6)--(12,6);

\draw[red, thick] (7,7)--(7,0);


\draw[blue, thick] (1,6)arc[start angle=180, end angle=0, radius=4];


\draw[blue, thick] (1,4)--(1,1) arc[start angle=180, end angle=270, radius=0.5]--(1.5,0.5)--(3,0.5) arc[start angle=90, end angle=0, radius=0.5];


\draw[blue, thick] (5,7) --(5,6);

\draw[blue, thick] (5,4)--(5,1)arc[start angle=180, end angle=270, radius=0.5]--(5.5,0.5)--(7,0.5) arc[start angle=90, end angle=0, radius=0.5];

\draw[blue, thick] (9,4)--(9,1) arc[start angle=180, end angle=270, radius=0.5]--(9.5,0.5)--(11,0.5) arc[start angle=90, end angle=0, radius=0.5];

\draw[blue, thick] (2,3)--(3.5,3)arc[start angle=90, end angle=0, radius=0.5]-- (4,1.5)arc[start angle=0, end angle=-90, radius=0.5]--(2,1)arc[start angle=270, end angle=180, radius=0.5]--(1.5,2.5)arc[start angle=180, end angle=90, radius=0.5];

\draw[blue, thick] (6.5,3)--(8,3)arc[start angle=90, end angle=0, radius=0.5]-- (8.5,1.5)arc[start angle=0, end angle=-90, radius=0.5]--(8,1) -- (6.5,1)arc[start angle=270, end angle=180, radius=0.5]--(6,2.5)arc[start angle=180, end angle=90, radius=0.5];

\draw[red, thick] (4,2.5)--(6,2.5)arc[start angle=90, end angle=-90, radius=0.5] --(4,1.5)arc[start angle=270, end angle=90, radius=0.5];

\draw[red, thick] (8,2.5)--(9.5,2.5)arc[start angle=90, end angle=-90, radius=0.5] --(8,1.5)arc[start angle=270, end angle=90, radius=0.5];


 \filldraw[black] (3,0.5) circle (2pt);
    \node at (2.5,0.25) {\( c_{2g} \)};

     \filldraw[black] (4,1.5) circle (2pt);
    \node at (4.5,1.1) {\( \widehat{c}_{4g-2} \)};

     \filldraw[black] (5,1.5) circle (2pt);
    \node at (5.3,1.2) {\( x_1 \)};
     \filldraw[black] (6,1.5) circle (2pt);
    \node at (6.4,1.2) {\( x_2 \)};
    
      \filldraw[black] (8.5,1.5) circle (2pt);
    \node at (8.7,1.2) {\( \widehat{c}_{4g-1} \)};

      \filldraw[black] (7,0.5) circle (2pt);
    \node at (6.8,0.25) {\({c}_1 \)};

      \filldraw[black] (11,0.5) circle (2pt);
    \node at (10.7,0.25) {\({c}_{2g} \)};

      \node at (2,2) {\({w} \)};

         \node at (3.75,2) {\({z} \)};

          \node at (6.25,2) {\({w} \)};

           \node at (8,2) {\(z \)};

            \node at (9.5,2) {\(w \)};

           \node at (0.5,0.5) {\(0 \)};

           \node at (4,0.5) {\(0 \)};

             \node at (8,0.5) {\(0 \)};

              \node at (11.5,2.5) {\(0 \)};

               \node at (6,11) {\(0 \)};

     \node at (2,3.5) {\(A \)};

     \node at (2,6.5) {\(A \)};

        \node at (3.25,2) {\(B\)};

          \node at (4.5,2) {\(C\)};

             \node at (4.5,3.5) {\(K\)};
              \node at (6,9) {\(K\)};
            \node at (5.5,2) {\(D\)};
              \node at (5.5,3) {\(E\)};

                \node at (5.5,6.5) {\(E\)};
              \node at (5.6,1) {\(G\)};

                  \node at (6.75,2) {\(F \)};
                  \node at (7.25,2) {\(H \)};
                   \node at (8.75,2) {\(I \)};

                    \node at (8.75,3.5) {\(K\)};
                  
                         \node at (10.5,2) {\(J \)};
                          \node at (10.5,6.5) {\(J \)};

\end{tikzpicture}
\caption{A neighborhood of $\Sigma_{P,1}$. Since the subscripts on the $z$ and $w$ basepoints play no role, we have suppressed them here. Part of the curve $\widehat{\alpha}_{4g-3}$ is not shown, since we know that the multiplicity of $D(\phi)$ does not change as we cross it. Note that neighborhoods of the points two points $c_{2g}$ are identified. The central region bounded above and below by dashed orange curves is the twisting region.}
     \label{fig:}
 \end{figure}
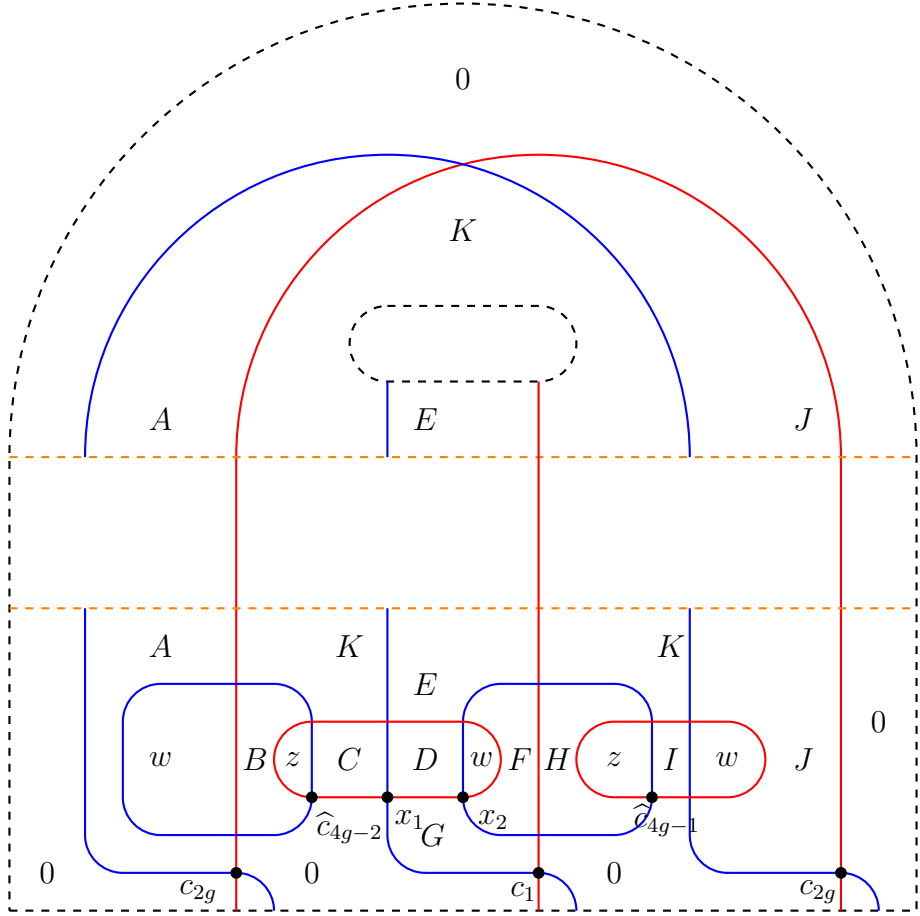

To begin, notice that no component of $\partial D(\phi)$ can pass along $\beta_G$ or $\alpha_G$ into $\Sigma_G$, since we have already shown that the multiplicity of the domains on either side of each curve are zero in the first part of the proof of this proposition. In particular, it follows that $E=K$. Here, and for the rest of the proof, the capital letters indicate the multiplicity of $\phi$ in the relevant elementary domain, as shown in Figure~\ref{fig:}. From the positioning of the basepoints near $c_{2g}$, the boundary of $\phi$ can traverse $\beta_{2g}$ at most once. In particular, the multiplicity in regions $A$ is at most one. It is easily seen that the multiplicity in region $A$ is either one or zero; i.e.  $A=0$ or $A=1$. We treat these two cases separately.

 \noindent\textbf{Case 1: $A=1$.} Observe that the component of $\partial D(\phi)$ out of $c_{2g}$ must pass directly upwards from the left hand-copy of $c_{2g}$ in Figure~\ref{fig:}. Since $\alpha_{2g}$ has a multiplicity zero domain to its left in the upper part of the Heegaard diagram, it follows that the component of $\partial D(\phi)$ must pass from $c_{2g}$ to the point directly above it on $\alpha_{2g}\cap\beta_{4g-3}$. It follows that region $K$ has multiplicity one. The positioning of the basepoint near $\widehat{c}_{4g-3}$ implies that $\partial D(\phi)$ must then proceed clockwise around $\widehat{\beta}_{4g-3}$ stopping on the first occasion it reaches $\widehat{c}_{4g-3}$. In particular we must have that $C=1$. Consider where $\partial D(\phi)$ goes next. Since the multiplicity in region $C$ is one, there are two options. The arc either stops at the first time passing through\begin{enumerate}
     \item[\textbf{Case 1a):}]   $x_1$ --- the point  in $\widehat{\alpha}_{4g-3}\cap\beta_{G}$ as shown in Figure~\ref{fig:}, or 
     \item[\textbf{Case 1b):}] $x_2$ --- the point in $\widehat{\beta}_{4g-2}\cap\widehat{\alpha}_{4g-3}$ as shown in Figure~\ref{fig:} --- so that $E=D=G+1$.
 \end{enumerate} In either case we have that \begin{equation}\label{eq:reuse}
     K=E=D=G+1.
 \end{equation} We consider these two subcases:

 \noindent\textbf{Case 1a):} Since the $\partial D(\phi)$ cannot pass into $\Sigma_{P,2}$ from $x_1$, we must have that the multiplicity of domains $D$ and $G$ are both one. Similarly, it follows that the multiplicities in region $F$ and $H$ are both zero and then in turn that the multiplicities in domain $I$ and $J$ are one and zero, respectively.  This completely specifies the shadow of $\phi$ when restricted to $\Sigma_P\cup\Sigma$. We call this homology class $\phi_1$.

 \noindent \textbf{Case 1b):} The positioning of the basepoint $w_{2g-1}$ implies that $D=E=1$, $K=H+1$ and $I=1$. The positioning of the basepoint near  $\widehat{c}_{4g-3}$ then implies that passing along $\widehat{\alpha}_{4g-3}$, $\phi$ must stop at the first intersection point to the right of $\widehat{\alpha}_{4g-3}$ at the first time of passing. Thus $J=0$. This and Equation~\ref{eq:reuse} determine the shadow of the disk in $\Sigma_P\cup S$. We call this class $\phi_2$.

  \noindent\textbf{Case 2: $A=0$.} It follows that the multiplicity in domain $C$ must also be zero by the positioning of the basepoint, we reduce to a case as shown in Figure~\ref{fig:movingalongboundary}. In particular we find that the multiplicities of all domains shown are zero, and so in turn that $\phi$ is constant on $\bm{c}\cup\widehat{\bm{c}}$. Suppose then that multiplicity in domain $K$ is one. It follows that the multiplicity in domains $B$ and $J$ are both one. Consider the behaviour of $\partial D(\phi)$ upon leaving $\widehat{c}_{4g-3}$. The positioning of the basepoints implies that $C=0$ and that $\partial D(\phi)$ must travel clockwise along $\widehat{\alpha}_{4g-3}$, terminating at on the first opportunity, either at a) $x_2$ in which case $D=G$, or b) $x_1$, in which case $G=D+1$. We analyse each of these two subcases.

 \noindent \textbf{Case 2.a):} Suppose we are in the first subcase. Observe that $\phi$ must be constant on $c_1$ as there are no possible intersection points for $\partial D(\phi)$ to travel to from $c_1$. Here we are abusing notation by allowing $c_1$ to denote the $c_j$ intersection point on the $\alpha$-curve corresponding to $a_s$; see Figure~\ref{fig:}. It follows that $G=0$, so that $K=E=F=H=1$. It is then easy to see from the positioning of the basepoints near $\widehat{c}_{4g-2}$ that $I=0$, and $J=1$ completely determining the domain of $\phi$ in the portion of the Heegaard diagram shown.   We call the homology class of the domain $\phi_3$.
 
   \noindent \textbf{Case 2.b):}  Suppose that we are in the second subcase. The positioning of the basepoint near $x_1$ implies that $G=1$ and that $F=H=1$. The positioning of the basepoint near $\widehat{c}_{4g-2}$ implies that $I=0$, fully determining the domain. We call the homology class of the domain $\phi_4$.

  The pseudoholomorphic disks from any $\bm{x}$ to $\widetilde{f}_{\tau^n}(\bm{y})$ that are not constant on the $c_i$ occur in pairs and therefore cancel modulo $2$. We defer the rigorous proof of this claim to Subsection~\ref{subsec:neck}, since it is rather involved. Consequently, if $\langle \bm{x}, \widetilde{f}_{\tau^n}(\bm{y})\rangle\neq 0$, then we have that $\bm{x}$ contains all of the intersection points $c_i$ and $\widehat{c}_i$. That is, $\bm{x}$ is in the image of $\widetilde{f}_{\tau^n}$, and we have the desired result.
  \end{proof}

We can now prove the following version of Proposition~\ref{prop:hatisiso}.

\begin{proposition}\label{prop:tildeisiso}
    For any $n$, the map $\widetilde{f}_{\tau^n}^*:\SFH(\mathcal{H}(-S))\to \widetilde{\HFK}(\widehat{\mathcal{H}_1}(-K_{\tau^n}),A_\text{min})$ is an isomorphism. In particular, $\widetilde{\HFK}(\mathcal{H}_1(-K_{\tau^n}),A_\text{min})$ is generated by contact class like generators.
\end{proposition}

\begin{proof}
    $\rank(\SFH(Y(S),\gamma(S)))=\rank(\widetilde{{\HFK}}(-K_{\tau^n} ,g(K)))$ by Juh\'asz's decomposition formula~\cite[Theorem 1.5]{juhasz2008floer}. The result follows from Lemma~\ref{lem:inj} and the definition of $\widetilde{f}_{\tau^n}$.
\end{proof}

We now modify $\mathcal{H}_1(-K_{\tau^n})$ to obtain $\mathcal{H}_2(-K_{\tau^n})$ as described in Section~\ref{sec:Heegaarddiagrams}. This requires only $\beta$-handle slides. For each handle-slide we count pseudo-holomorphic triangles as in Figure~\ref{fig:standardtraingles}. We call the composition of all of these maps $g_*$.

\begin{lemma}\label{lem:h1toh2}
     $g_*:\widetilde{\HFK}(\mathcal{H}_1(-K_{\tau^n}))\to \widetilde{\HFK}(\mathcal{H}_2(-K_{\tau^n}))$ is a graded isomorphism that sends contact class like elements to contact class like elements.
\end{lemma}

\begin{proof}
   Recall from the proof of the independence of Heegaard Floer homology that handleslides induce isomorphisms via counts of holomorphic triangles subject to appropriate boundary conditions~\cite{ozsvath2004holomorphic}. These are exactly the triangles that our map $g_*$ is counting, so the fact that $g_*$ is a graded isomorphism follows.

   Near each $c_i$ we have a local picture as in Figure~\ref{fig:standardtraingles}, it follows that $g_*$ sends contact class like generators to contact class like generators.
\end{proof}

Now we repeatedly remove pairs of $\bm{\widehat{\alpha}}$- and $\bm{\widehat{\beta}}$-curves in $\mathcal{H}_2(-K_{\tau^n})$ to obtain $\widehat{\mathcal{H}}(-K_{\tau^n})$.

\begin{lemma}\label{lem:h2tohat}
    The map $h:\T_{\bm{\alpha}}\cap\T_{\bm{\beta}}\mapsto \T_{\bm{\alpha}}\cap\T_{\bm{\beta}}$ obtained by forgetting the elements in the curves $\widehat{\bm{\alpha}},\widehat{\bm{\beta}}$ extends linearly to a chain map from the span of the contact class like generators in $\widetilde{\CFK}(\mathcal{H}_2(-K_{\tau^n}),A_\text{min})$ to $\widehat{\CFK}(\widehat{\mathcal{H}}(-K_{\tau^n}),A_{\text{min}})$. Moreover, $h_*$ is an isomorphism on homology. In particular, $\widehat{\HFK}(\widehat{\mathcal{H}}(-K_{\tau^n}),A_{\text{min}})$ is generated by contact class like elements.
\end{lemma}

\begin{proof}
The fact that $h$ is a chain map follows immediately from the positioning of the basepoints near the $c_i$ intersection points. Observe that there is in fact a unique element of $\widehat{\alpha}_i\cap\widehat{\beta}_i$ that is contained in any generator in $\T_{\bm{\alpha}}\cap\T_{\bm{\beta}}$ of the correct Alexander grading. $h$ has a right inverse obtained by the map obtained as the linear extension of the map given by appending this intersection point to any generator of $\T_{\bm{\alpha}}\cap\T_{\bm{\beta}}$. The positioning of the basepoints implies that this is a chain map. It follows that $h_*$ is injective. Surjectivity on homology once again follows Juh\'asz's decomposition formula~\cite{juhasz2008floer}. Thus we have the desired isomorphism in particular, $\widehat{\HFK}(\widehat{\mathcal{H}}(-K_{\tau^n}),A_{\text{min}})$ is freely generated by contact class like elements.
\end{proof}

We can now prove Proposition~\ref{prop:hatisiso}.

\begin{proof}[Proof of Proposition~\ref{prop:hatisiso}]
    $\hat{f}_{\tau^n}$ is a chain map that has the same image as the composition of chain maps $h\circ g\circ\tilde{f}_{\tau^n}$. Since $\tilde{f}_{\tau^n}$, $g$, and $h$ induce isomorphisms on homology by Proposition~\ref{prop:tildeisiso}, Lemma~\ref{lem:h1toh2} and Lemma~\ref{lem:h2tohat} respectively, it follows that  $\hat{f}_{\tau^n}$ also induces an isomorphism on homology, as desired. Note, moreover, that $\widehat{f}_{\tau^n}$ and $h\circ g\circ\tilde{f}_{\tau^n}$ map every element in $\CF(\mathcal{H}(-S))$ to a contact class like element in $\widehat{\HFK}(\widehat{\mathcal{H}}(-K_{\tau^n}))$, so the result follows.
\end{proof}

\subsection{Neck Stretching}\label{subsec:neck}
In this subsection we fill the gap in the proof of Lemma~\ref{lem:inj}. The proof of this proposition is rather involved, requiring the use of a neck-stretching argument. For a more detailed description of this procedure we refer the reader to~\cite[Section 12]{Lipshitzcylindrical}. To apply this technique we begin with a small digression into Lipshitz's cylindrical reformulation of Heegaard Floer homology~\cite{Lipshitzcylindrical}.

In the cylindrical setting, the Heegaard Floer differential counts pseudo-holomorphic curves in the four manifold $W:=\Sigma\times[0,1]\times \R$ subject to various technical requirements. To state them, we set some notation: let $\mathbb{D}$ denote the $[0,1]\times\R$ factor of $\Sigma\times[0,1]\times\R$, and let $\pi_\R:W\to \R$, $\pi_\mathbb{D}:W\to\mathbb{D}$, and $\pi_\Sigma:W\to\Sigma$ denote projections onto the relevant factors of $W$. The cylindrical differential then counts pseudo-holomorphic maps $u:S\to W$ where $S$ is a surface with boundary and $2g$ punctures on $\partial S$, $\{p_1,\dots,p_g,q_1,\dots q_g\}$ and the following conditions are satisfied.

\begin{enumerate}
    \item[\textbf{M0}] The source is smooth.
    \item[\textbf{M1}] $u(\partial S)\subset C_\alpha\cup C_\beta$, where here $C_\alpha:=\bm{\alpha}\times\{0\}\times\R$, $C_\beta:=\bm{\beta}\times\{1\}\times\R$.
    \item[\textbf{M2}] The energy (as defined in~\cite[Section 5.3]{compactnesssymplecticfieldtheory}) of $u$ is finite.
    \item[\textbf{M3}] For each $i$, $u^{-1}(\alpha_i\times\{1\}\times\R)$ and  $u^{-1}(\beta_i\times\{0\}\times\R)$ consist of exactly one component of $\partial  S\setminus\{p_1,\dots,p_g,q_1,\dots q_g\}$.
    \item[\textbf{M4}]$\lim_{w\to p_i}\pi_\R\circ u(w)=-\infty$ and $\lim_{w\to q_i}\pi_\R\circ u(w)=\infty$.\item[\textbf{M5}] There are no components on which $\pi_\mathbb{D}\circ u$ is constant.
\end{enumerate} Let $\phi\in\pi_2(\bm{x},\bm{y})$. We let $\mathcal{M}(S,\phi)$ denote the moduli-space of pseudo-holomorphic curves of homology class $\phi$ with source $S$. Let $p$ be a point in $\Sigma$. Let $\rho^p:\mathcal{\phi}\to\Sym^{n_p}(\D)$ be the map $u\mapsto (\pi_\D\circ u)\circ(\pi_\Sigma\circ u)^{-1}(p)$
Suppose $\phi\in\pi_2(\bm{x},\bm{y})$ is a homology class and $p\in\Sigma\setminus(\bm{\alpha}\cup\bm{\beta})$. Let $X$ denote the smooth manifold avoiding the fat diagonal, i.e. the points in $\Sym^{n_p}(\D)$ at least two elements of which agree. We consider the matched moduli space $\mathcal{M}(S,\phi,X):=\{u\in\mathcal{M}(S,\phi):\rho^p(u)\in X\}$.

\begin{proposition}[{\cite[Proposition 2.6]{zemkedualitymapping}}]\label{prop:indexformula}
If $X\in\Sym^n(\mathbb{D})$ is a submanifold which avoids the fat diagonal, then $\mathcal{M}(S,\phi,X)$ is a smooth manifold of dimension \begin{align}
    \ind(u)=\mu(\phi)-2\sing(u)-\codim(X)
\end{align} near any curve $u$ satisfying conditions \textbf{M0} through \textbf{M5}.
\end{proposition}

 A holomorphic curve whose boundary is mapped only to $C_\alpha$, is called a \emph{cylindrical $\bm{\alpha}$-degeneration.} \emph{Cylindrical $C_\beta$-degenerations} are defined similarly.

We fix some more notation. Consider a point $\bm{x}\in\T_{\bm{\alpha}}\cap\T_{\bm{\beta}}\cap\im(f_{\tau^n})$. Let $\bm{y}\in\T_{\bm{\alpha}}\cap\T_{\bm{\beta}}$ be another point. By the argument in the proof of proposition above the moduli space of Maslov index one curves from $\bm{y}$ to $\bm{x}$, $\mathcal{M}(\phi)$ splits as a disjoint union of $\mathcal{M}(\phi_1)\sqcup\mathcal{M}(\phi_2)\sqcup \mathcal{M}(\phi_3)\sqcup\mathcal{M}(\phi_4)$, where $\mathcal{M}(\phi_i)$ indicates the moduli space of Maslov index one curves defined in the proof of Lemma~\ref{lem:inj} (one at the end of each of the cases \textbf{1.a)}, \textbf{1.b)}, \textbf{2.a)}, \textbf{2.b)}). As standard, we let $\widehat{\mathcal{M}}(\phi_i)$ indicate the quotient of ${\mathcal{M}}(\phi_i)$ by the usual $\R$-action.

 We are now in a position to fill the gap in the proof of Lemma~\ref{lem:inj}. This is the only point in the paper where we will make use of the hypothesis that $K$ has $w$-avoiding exterior in an essential way.

\begin{proposition}\label{prop:neck}
    Let $\phi_i$ denote the four homology classes of curves defined in the proof of Lemma~\ref{lem:inj}. We have that; \begin{enumerate}
        \item  $|\widehat{\mathcal{M}}(\phi_1)|=|\widehat{\mathcal{M}}(\phi_3)|=0$,  while 
        \item $|\widehat{\mathcal{M}}(\phi_2)|\equiv|\widehat{\mathcal{M}}(\phi_4)|\mod 2$.   \end{enumerate}
\end{proposition}

Consider the curve $c$ as shown in Figure~\ref{fig:neck}. The idea of the proof is to relate the moduli spaces of holomorphic representatives in a complex structure with sufficiently stretched neck at $c$ with the moduli space of the limits and deduce the proposition from there. It turns out that in the limit the Heegaard diagram can be identified with the Heegaard diagram shown in Figure~\ref{fig:neck}.

 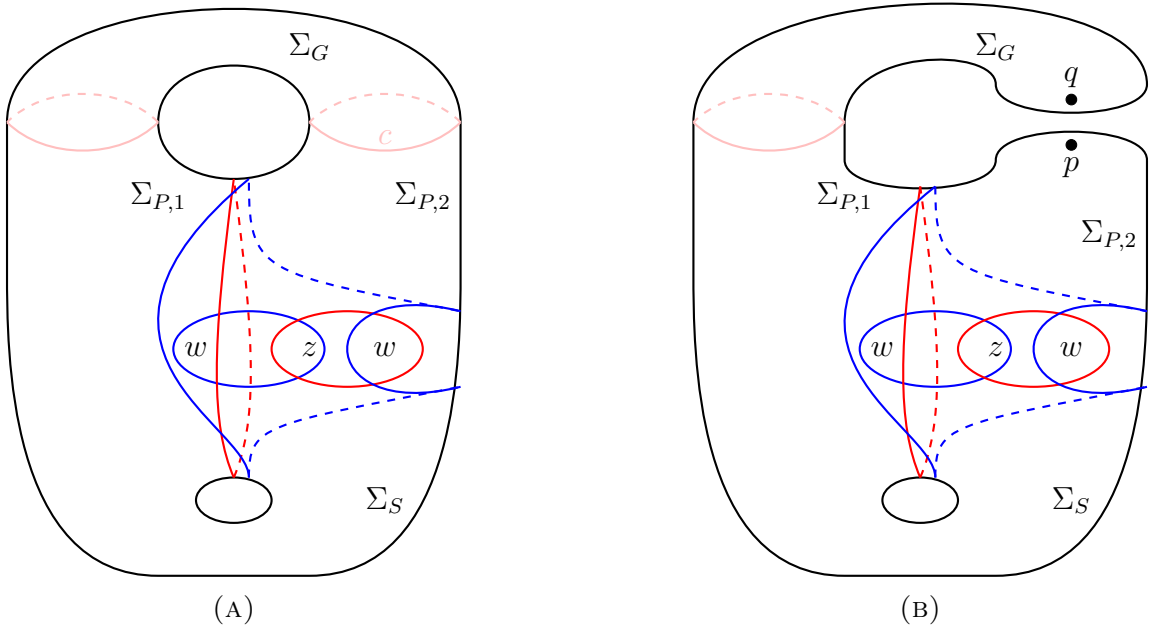
\begin{figure}
    \centering

\begin{subfigure}{0.45\linewidth}
\centering
   \begin{tikzpicture}


\draw[ thick]
(3,3) -- (3,1)
.. controls (3,0) and (3,-3) ..
(1,-3) -- (-1,-3)
.. controls (-3,-3) and (-3,0) ..
(-3,1) -- (-3,3) .. controls (-3,5) and (3,5) .. (3,3);

\draw[ thick](1,3)
.. controls (1,2) and (-1,2) ..
(-1,3) ;
\draw[ thick](1,3)
.. controls (1,4) and (-1,4) ..
(-1,3) ;

\draw[pink, thick,dashed] (3,3) 
.. controls (2.5,3.5) and (1.5,3.5) ..
(1,3);
\draw[pink, thick] (3,3) 
.. controls (2.5,2.5) and (1.5,2.5) ..
(1,3);

\draw[pink, thick, dashed] (-1,3) 
.. controls (-1.5,3.5) and (-2.5,3.5) ..
(-3,3);
\draw[pink, thick] (-1,3) 
.. controls (-1.5,2.5) and (-2.5,2.5) ..
(-3,3);;

\draw[draw=black, thick] (0,-2) ellipse (0.5 and 0.3);

\draw[draw=blue, thick] (0.2,0) ellipse (1 and 0.5);

\draw[draw=red, thick] (1.5,0) ellipse (1 and 0.5);


\draw[red, thick] (0,-1.7)
.. controls (-0.3,-1) and (-0.3,0) .. (0,2.25);
\draw[red, thick, dashed] (0,-1.7)
.. controls (0.3,-1) and (0.3,0) .. (0,2.25);

\draw[blue, thick] (0.2,-1.7)
.. controls (0.2,-1) and (-2.5,0) .. (0.2,2.25);
\draw[blue, thick, dashed] (0.2,-1.7)
.. controls (0.2,-1) and (0.4,-1) .. (3,-0.5);

\draw[blue, thick] (3,-0.5)
.. controls (1,-1) and (1,1) .. (3,0.5);

\draw[blue, thick, dashed] (0.2,2.2)
.. controls (0.2,1) and (0.4,1) .. (3,0.5);


        \node at (2,-2) {\(\Sigma_S \)};

             \node at (1,4) {\(\Sigma_G \)};

        \node at (2.5,2) {\(\Sigma_{P,2} \)};

         \node at (-1,2) {\(\Sigma_{P,1} \)};

              \node at (-0.5,0) {\(w \)};
                  \node at (1,0) {\(z \)};
                   \node at (2,0) {\(w \)};

                   \node[pink] at (2,2.8) {\(c \)};

\end{tikzpicture}
\caption{}\label{subfig:neck1}
\end{subfigure}
\hfill\begin{subfigure}{0.45\linewidth}
\centering
   \begin{tikzpicture}


\draw[ thick]
(3,2.5) -- (3,1)
.. controls (3,0) and (3,-3) ..
(1,-3) -- (-1,-3)
.. controls (-3,-3) and (-3,0) ..
(-3,1) -- (-3,3) .. controls (-3,5) and (3,5) .. (3,3.5)
.. controls (3,3) and (1,3) .. (1,3.5)
.. controls (1,4) and (-1,4) .. (-1,3)--(-1,2.5)
.. controls (-1,2) and (1,2) .. (1,2.5)

.. controls (1,3) and (3,3) .. (3,2.5)
;


\draw[pink, thick, dashed] (-1,3) 
.. controls (-1.5,3.5) and (-2.5,3.5) ..
(-3,3);
\draw[pink, thick] (-1,3) 
.. controls (-1.5,2.5) and (-2.5,2.5) ..
(-3,3);;

\draw[draw=black, thick] (0,-2) ellipse (0.5 and 0.3);

\draw[draw=blue, thick] (0.2,0) ellipse (1 and 0.5);

\draw[draw=red, thick] (1.5,0) ellipse (1 and 0.5);


\draw[red, thick] (0,-1.7)
.. controls (-0.3,-1) and (-0.3,0) .. (0,2.15);
\draw[red, thick, dashed] (0,-1.7)
.. controls (0.3,-1) and (0.3,0) .. (0,2.15);

\draw[blue, thick] (0.2,-1.7)
.. controls (0.2,-1) and (-2.5,0) .. (0.2,2.15);
\draw[blue, thick, dashed] (0.2,-1.7)
.. controls (0.2,-1) and (0.4,-1) .. (3,-0.5);

\draw[blue, thick] (3,-0.5)
.. controls (1,-1) and (1,1) .. (3,0.5);

\draw[blue, thick, dashed] (0.2,2.15)
.. controls (0.2,1) and (0.4,1) .. (3,0.5);


        \node at (2,-2) {\(\Sigma_S \)};

             \node at (1,4) {\(\Sigma_G \)};

        \node at (2.5,1.5) {\(\Sigma_{P,2} \)};

         \node at (-1,2) {\(\Sigma_{P,1} \)};

              \node at (-0.5,0) {\(w \)};
                  \node at (1,0) {\(z \)};
                   \node at (2,0) {\(w \)};

                        \filldraw[black] (2,2.7) circle (2pt);
                      \node at (2,2.4) {\(p \)};
                      
                        \filldraw[black] (2,3.3) circle (2pt);
                           \node at (2,3.6) {\(q \)};

\end{tikzpicture}
\caption{}\label{subfig:neck2}
\end{subfigure}

    \caption{The process of neck stretching replaces pseudo-holomorphic curves in a $4$-manifold obtained from the left hand Heegaard diagram with one obtained from the right hand Heegaard diagram. We suppress the subscripts on the $z$ and $w$-basepoints, since they are unimportant.} \label{fig:neck}
\end{figure}

\begin{proof}
Consider a sequence of almost complex structures $J(T_k)$ with $T_k\to\infty$. Observe that in the limit the homotopy class $\phi$ decomposes as a sum of homotopy classes $\phi_{i,g}+\phi_{i,p}$, where $\phi_{i,g}$ is the component of $\phi_i$ that lives in the Heegaard diagram for the top piece --- $\Sigma_G$ with one of its two boundary components pinched to a point --- and $\phi_{i,p}$ is the component of $\phi$ that lives outside of this region of the Heegaard diagram --- i.e. in $\Sigma_S\cup\Sigma_P$ with one of its boundary components pinched to a point.  
Observe that  \begin{equation}\label{eq:maslov}\mu(\phi_i)=1=\mu(\phi_{i,g})+\mu(\phi_{i,p})-2n_p(\phi_i).\end{equation}here $p$ is a point as shown in Figure~\ref{fig:neck} and $\phi_{i,g}$ and $\phi_{i,p}$ are the components of $\phi_i$ in the limit.

We first show that $|\widehat{\mathcal{M}}(\phi_1)|\equiv|\widehat{\mathcal{M}}(\phi_3)|=0$. To see this observe that $n_p(\phi_i)=1$, while $\mu(\phi_{i,p})=2$, so that $\mu(\phi_{i,g})=1$. The assumption on $\mathcal{H}$ coming from the requirement that $K$ has $w$-avoiding exterior rules out the existence of homotopy classes of disks with $\mu(\phi_{i,g})=1$, concluding the proof of the first part of the proposition.

We now proceed to the second part of the proposition: $|\widehat{\mathcal{M}}(\phi_2)|\equiv|\widehat{\mathcal{M}}(\phi_4)|\mod 2$. We will first show that the moduli space of each $\phi_i$ can be identified with the moduli space of pairs of curves satisfying a matching condition. Since the proof is same for each $i$, we suppress the index $i$ from our notation.  We readily see that $\mu(\phi_p)=1$, and $n_p=1$ so that $\mu(\phi_g)=2$.

    For any sequence of curves representing $\phi$ for the sequence of almost complex structures $J(T_k)$ with $T_k\to\infty$, we can extract limiting curves $U_g$ and $U_p$ such that $U_g$ and $U_p$ have components $u_g$ and $u_p$ which satisfy the matching condition $\rho^p(u_g)=\rho^q(u_p)$.  Here $q$ is the point shown in Figure~\ref{fig:neck}. Since $n_p=1$, this is an equality of sets with at most one element. This implies that $u_p$ and $u_g$ each have a single connected component.

We wish to rule out a number of possibilities of curves that can occur in the limit. Sphere bubbling cannot occur because our curves are required to avoid the basepoint, since we are working with the hat flavour of knot Floer homology; see Figure~\ref{fig:neck}. Similarly, cylindrical $\alpha$ or $\beta$-boundary degenerations cannot occur, because we are only interested in moduli spaces of curves which avoid the basepoints, and the unique connected region in the complement of $\alpha$ or $\beta$-curves contains a basepoint; see Figure~\ref{fig:neck}.

Let $S_G$ denote the source of the curve $u_g$, $\phi'_g$ denote the homotopy class of the image of the curve and $S_P$ denote the source of $u_p$. Set ${X:=\{\rho^p(u_p):u_p\in\mathcal{M}(\phi_p)\}\subset \Sym^n(\mathbb{D})}$. This has codimension $2n_p-1=1$. Note that $0\leq\mu(\phi'_g)\leq\mu(\phi_g)\leq 2$.

The curve $u_g$ satisfies conditions (M1)-(M5). All but condition M5 are readily verified. M5 holds because $u_g$ necessarily has boundary with punctures --- since we are working with the hat flavour of Heegaard Floer homology, so that $u_g$ cannot be closed or an $\alpha$ or $\beta$ boundary degeneration --- so in turn $u_g$ being locally non-constant is forced by (M4).  By Proposition~\ref{prop:indexformula} we then have that near $u_g$: \begin{align*}
\dim(\mathcal{M}(S_G,\phi_g',X))&=\mu(\phi_g')-\codim(X)-2\sing(u_g)\\&\leq 1-2\sing(u_g)\\&\leq 1.\end{align*}

On the other hand ${\dim(\mathcal{M}(S_g,\phi'_g,X))\geq 1}$, since $\D$ carries an action of a $1$-dimensional subgroup of $\PSL(2,\R)$ on the target preserving the matching condition and the other boundary conditions. Thus we have that $2\sing(u_g)=0$, so that $u_g$ is embedded, and $\mu(\phi'_g)=2$.  

We cannot have Maslov index $0$ connected components of $\phi_g$ since otherwise the expect dimension for the moduli space will be $0$ for example see \cite[Corollary 7.2]{Lipshitzcylindrical}, but the dimension should be at least one because of the action of a one dimensional subgroup of $\PSL(2,\R)$ as above. Thus, $\mu(\phi'_g)=\mu(\phi_g)$ implies that $u_g$ must be the only component of $U_g$. Similarly, as we have seen $\mu(\phi_p)=1$ implies $U_p$ must have one connected component $u_p$. By inspecting the Heegaard diagram we readily see that $\mathcal{M}(\phi_p)$ is one dimensional, with a single connected component.

Define $\mathcal{M}(\phi_g,X)$ to be $\{u\in\mathcal{M}(\phi_g):\rho^q(u)\in X\}$. We have constructed an injective map from $\mathcal{M}(\phi)$ to pairs of maps $\mathcal{M}(\phi_g,X)\times\mathcal{M}(\phi_p)\subset\mathcal{M}(\phi_g)\times\mathcal{M}(\phi_p)$; i.e. pairs of curves satisfying a matching condition. We wish to show this map is a bijection. Since the moduli spaces in question have a finite number of components it suffices to provide an injective map in the other direction. Such a map is provided to us by a result of Lipshitz.
If the almost complex structure achieves transversality at $u_p$ and $u_g$ --- which we may assume --- then there is a neighborhood $U$ of $(u_p,u_g)$ in the compactification of holomorphic disks in $\Sigma\times[0,1]\times\R$ such that $$U\cap\big(\{(u_p,u_g)\}\cup \bigcup_{T>0}\widehat{\mathcal{M}}_{J(T)}(\phi_i\big))$$ is diffeomorphic to $(0,1]$. This follows by modifying the proof of~\cite[Proposition A.2]{Lipshitzcylindrical}, just as is required to prove~\cite[Proposition A.3]{Lipshitzcylindrical}.

To conclude the proof of the Proposition we want to add the index $i$ back to our notation and verify that the size of the moduli space is independent of the matching. More precisely, observe that for each homology class $\phi_i$, the corresponding component $\phi_g$ is the same. On the other hand $\phi_p$ depends on $i$, so we denote it by $\phi_p^i$, but as we have seen $|\widehat{\mathcal{M}}(\phi_p^i)|=1$. Thus, to conclude the result we only need to show that $|\widehat{\mathcal{M}}(\phi_g,X_i)|$ is independent of $i$, where we recall that $X_i=\{\rho^{p}(u):u\in\mathcal{M}(\phi_p^i)\}$. This follows by adapting a corresponding argument in the proof of~\cite[Proposition 5.3]{zemkequasistabbasepoint}.

Set $X_i:=\{\rho^{p}(u):u\in\mathcal{M}(\phi_p^i)\}$ and $X_j:=\{\rho^{p}(u):u\in\mathcal{M}(\phi_p^j)\}$ respectively. Let $X_t$ be subsets of $\D$ interpolating between $X_i$ and $X_j$. Consider the one dimensional space $$\widehat{\mathcal{M}}(\phi_g,[i,j]):=\underset{t\in[i,j]}{\bigcup}\widehat{\mathcal{M}}(\phi_g,X_t).$$

 We wish to show the only boundary components are $\widehat{\mathcal{M}}(\phi_g,X_i)$ and $\widehat{\mathcal{M}}(\phi_g,X_j)$, so in turn that these moduli spaces are in bijection. The only other possible ends of this moduli space correspond to strip breaking, sphere bubbling, cylindrical $\alpha$ or $\beta$ boundary degenerations, or annoying curves (i.e. curves for which there is a non-empty open subset of $S$ for which $\pi_\mathbb{D}\circ u:S\to W$ is constant). Sphere bubbling and $\alpha$ and $\beta$ boundary degenerations cannot happen for the same reasons as earlier in the proof; the presence of basepoints prevents such behavior, since we are working with the hat flavour of Heegaard Floer homology. 

 To rule out annoying curves and strip breaking we can again proceed as before. Note that the Maslov index of any limit curve must be the same as the Maslov index of any curve in the interior of the moduli space i.e. two. To exclude strip breaking, note that since the Maslov index is non-negative and $0$ only for constant curves, there must be two components to the limit split curve, each of which has Maslov index one. Let $u'$ be the matched curve, $\psi'$ be its homology class and $S'$ be its source.  Now, appealing to Proposition~\ref{prop:indexformula} once more, we see that; \begin{align*}\dim(\mathcal{M}(S',\psi',X_i))&=\mu(\psi')-\codim(X_i)-2\sing(u')\\&\leq 1-1-0=0\end{align*} but $\mathcal{M}(S',\psi',X_i)$ carries an action of a one dimensional subgroup of $\PSL(2,\R)$ as above, so the dimension of $\mathcal{M}(S',\psi',X_i)$ is at least one, a contradiction. Annoying curves are ruled out because the Maslov index would also be pushed above two.
\end{proof}

\begin{remark}\label{rem:hypandgeneral}
  The hypothesis that $K$ has $w$-avoiding exterior is  only used in proving part $1$ of Proposition~\ref{prop:neck}. We do not know if this is necessary for the statement to be true, or indeed for the weaker claim that $|\widehat{\mathcal{M}}(\phi_1)|\equiv|\widehat{\mathcal{M}}(\phi_3)|\mod 2$, which would also have sufficed for our purposes.
\end{remark}

\subsection{Generators for arbitrary diffeomorphisms}\label{subsec:arbitrarydiffeos}

In this section we generalize Proposition~\ref{prop:hatisiso} to the case of arbitrary diffeomorphisms, namely;

\begin{proposition}\label{prop:satiatedimpliesgeneration}
    Suppose $K$ has $w$-avoiding exterior. Then $\widehat{\HFK}(-K,A_\text{min})$ is generated by contact class like generators for an appropriate choice of Heegaard diagram for $(G,\gamma)$ and basis of arcs.
\end{proposition}

In the next section, we will show that this conclusion holds for arbitrary choices of Heegaard diagram and basis of arcs.  We do not know if the Proposition fails without the assumption that $K$ has $w$-avoiding exterior. 

 Let $\rho$ denote a negative Dehn twist about a non-separating simple closed curve $c$ in $R_+(\gamma(S))$. Given a Heegaard diagram $\widehat{\mathcal{H}}(-K_{\psi})$ we can obtain a Heegaard diagram for $\widehat{\mathcal{H}}(-K_{\rho\circ\psi})$ by applying a right handed Dehn twist about $c$, viewed as a curve in $\Sigma'$.  We define a map $\rho_*:\widehat{\HFK}(\widehat{\mathcal{H}}(-K_\psi),A_\text{min})\to \widehat{\HFK}(\widehat{\mathcal{H}}(-K_{\rho\circ\psi}),A_\text{min})$ as a count of pseudo-holomorphic triangles as in~\ref{fig:triangle}.

\begin{proposition}\label{prop:generorsaftertwists}
 $\rho_*$ is an isomorphism. Moreover, if $\widehat{\HFK}(\mathcal{H}_\psi,A_\text{min})$ is generated by contact class like elements then so too is $\widehat{\HFK}(\mathcal{H}_{\rho\circ\psi},A_\text{min})$.
\end{proposition}

Here we are abusing notation as usual by allowing $c_i$ to refer to distinct points in distinct Heegaard diagrams. We follow Baldwin and Vela-Vick's strategy in the proof of~\cite[Theorem 2.3]{baldwin_note_2018}. Similar techniques have been used by Honda-Kazez-Mati\'c; see, for example, the proof of~\cite[Section 3.4]{honda2006contact}.

\begin{proof}

   Observe that the map $\rho_*$ fits into the following surgery exact triangle:

\begin{equation*}
\centering
    \begin{tikzcd}      \widehat{\HFK}(\widehat{\mathcal{H}}(-K_\psi),A_{\text{min}})\ar[rr,"\rho^*"]&&\widehat{\HFK}(\widehat{\mathcal{H}}({-K_{\psi\circ \rho}),A_{\text{min}}})\ar[dl]\\&\widehat{\HFK}(\mathcal{H}_0,A_{\text{min}})\ar[ul]&
    \end{tikzcd}
    \end{equation*}

    Here $\mathcal{H}_0$ is a Heegaard diagram for the knot in the three manifold obtained by doing $0$-surgery on $c$ (with respect to the framing induced by $\Sigma$). This knot is of genus strictly less than that of $K_\psi$ and $K_{\psi\circ\rho}$, so that $\widehat{\HFK}(\mathcal{H}_0,A_\text{min})=0$. Now, $\rho^*$ is a triangle counting map for $\bm{\theta}$. Of course, this is well defined since the positioning of the basepoints imply that $\bm{\theta}$ is a cycle in $\widehat{\CFK}(\mathcal{H}_0)$. It suffices to show that the triangle counting map is constant on each $c_i$ and $\widehat{c}_i$. This follows directly from a careful consideration of the location of basepoints in a neighborhood of each $c_i$; see Figure~\ref{fig:standardtraingles}.\end{proof}

\begin{proof}[Proof of Proposition~\ref{prop:satiatedimpliesgeneration}]
 Recall that any diffeomorphism $R_+(\gamma(S))\to R_+(\gamma(S))$ can be obtained from $\tau^n$ by post composing with a sequence of negative Dehn twists about essential simple closed curves in $R_+(\gamma(S))$. The result now follows from Proposition~\ref{prop:hatisiso} and repeated applications of Proposition~\ref{prop:generorsaftertwists}
\end{proof}

 \begin{remark}\label{rem:generalhard}
    It would be natural to attempt to prove Conjecture~\ref{con:sivek} for knots admitting Seifert surfaces with more complicated exteriors; namely exteriors that can be decomposed along arbitrary families of product disks and product annuli (as opposed to exactly two product annuli as in Definition~\ref{def:avoidext}) into product and non-product sutured manifold pieces (more general than the sutured exteriors of $w$-avoiding knots). However, it is not clear to the authors how to generalize the proof of Proposition~\ref{prop:satiatedimpliesgeneration} to this setting.
\end{remark}

\subsection{Independence of Auxiliary choices.}\label{subsec:independence}

Throughout this section we have suppressed the dependence of our notation on the choice of basis of arcs for $R_+(\gamma_0)$ in our Heegaard diagram notation, as well as the choice of Heegaard diagram for $(Y,\gamma_0)$.
In this subsection, we show that this is justifiable. A stronger version of this independence will be used in Section~\ref{sec:nexttominimal}.

To that end we recall that in Definition $\ref{def: H(sigma,alphaG,betaG,a)}$ $\widehat{\mathcal{H}}(\Sigma_G,\bm{\beta_G},\bm{\alpha_G},\bm{{b}},\bm{a},\bm{\overline{a}})$ denotes the Heegaard diagram obtained in Section~\ref{subsubsec:HFhat}. 
To show that the property of $$\widehat{\HFK}(\widehat{\mathcal{H}}(\Sigma_G,\bm{\beta_G},\bm{\alpha_G},\bm{b},\bm{a},\bm{\overline{a}}),[-S],-g)$$ being generated by contact class like elements is independent of the choice of $(\Sigma_G,\bm{\beta_G},\bm{\alpha_G},\bm{{b}},\bm{a},\bm{\overline{a}})$ it suffices to show that the maps used in the proof of the invariance of Heegaard Floer homology send contact class like elements to contact class like elements. That is, we have to show the maps corresponding to the following moves send contact class like elements to contact class like elements:

\begin{enumerate}
    \item\label{item:arcslide} For a pair of arcs $\overline{a}_i,\overline{a}_j\in\bm{\overline{a}}$ with adjacent endpoints in $\partial\Sigma$ on the complement of the $z$ and $w$ sliding the arc $\overline{a}_i$ over $\overline{a}_j$ and the arc $\overline{b}_i$ over $\overline{b}_j$.
    \item\label{item:arccurveslide} Sliding an arc $a_i$ over a curve $\alpha_j\in\bm{\alpha_G}$.
      \item\label{item:arccurveslidebeta} Sliding an arc $b_i$ over a curve $\beta_j\in\bm{\beta_G}$.
    \item Sliding a curve $\alpha_1\in\bm{\alpha_G}$ over a curve $\alpha_2\in\bm{\alpha_G}$.
      \item Sliding a curve $\beta_1\in\bm{\beta_G}$ over a curve $\beta_2\in\bm{\beta_G}$.
    \item\label{item:stab} Isotopies of the $a_i$ arcs, $b_i$ arcs, as well as the curves in $\bm{\alpha_G}\cup\bm{\beta_G}$.
    \item\label{item:stabguts} Stabilizing $(\Sigma_G,\bm{\beta_G},\bm{\alpha_G})$.
\end{enumerate}

\begin{lemma}
Suppose the arcs $\bm{\overline{a}'}$ are obtained from the arcs $\bm{\overline{a}}$ by an arc slide. If contact class like elements generate $\widehat{\HFK}(\widehat{\mathcal{H}}(\Sigma_G,\bm{\beta_G},\bm{\alpha_G},\bm{{b}},\bm{a}, \bm{\overline{a}}),[-S],-g))$ then contact class like elements generate $\widehat{\HFK}(\widehat{\mathcal{H}}(\Sigma,\bm{\beta_G},\bm{\alpha_G},\bm{b},\bm{a}, \bm{\overline{a}'}),[-S],-g)$.
\end{lemma}

\begin{proof}
    This follows just as in the proof of~\cite[Lemma 3.5]{honda2006contact}.
\end{proof}

\begin{lemma}\label{lem:arcovercurve}
Suppose the arcs $\bm{a'}$ are obtained from the arcs $a$ by sliding an arc $a\in\bm{a} $  over a curve $\alpha\in\bm{\alpha_G}$. If contact class like elements generate $\widehat{\HFK}(\widehat{\mathcal{H}}(\Sigma,\bm{\beta_G},\bm{\alpha_G},\bm{b},\bm{a},\bm{\overline{a}}),[-S],-g)$ then contact class like elements generate $\widehat{\HFK}(\widehat{\mathcal{H}}(\Sigma,\bm{\beta_G},\bm{\alpha_G},\bm{b},\bm{a'},\bm{\overline{a}}),[-S],-g)$. A similar statement follows for sliding  an arc $b\in\bm{b} $  over a curve $\beta\in\bm{\beta_G}$.
\end{lemma}

\begin{proof}
 To show contact class like generators are preserved by these operations, we need only concern ourselves  with local behavior of the curves in $\Sigma_S$. $\Sigma\setminus\Sigma_{G\cup P}$, noting that the arcs in $\Sigma\setminus\Sigma_{G\cup P}$ are unchanged under this operation. The result then follows from considering the local behavior of pseudo-holomorphic disks in a neighborhood of the $c_i$ intersection points.
\end{proof}

\begin{lemma}
Suppose $\bm{\alpha_G'}$ and $\bm{\beta_G'}
$ are obtained from $\bm{\alpha_G}$ and  $\bm{\beta_G}
$ by handleslides. If contact class like elements generate $\widehat{\HFK}(\widehat{\mathcal{H}}(\Sigma,\bm{\beta_G},\bm{\alpha_G},\bm{b},\bm{a},\bm{\overline{a}}),[-S],-g))$ then contact class like elements generate $\widehat{\HFK}(\widehat{\mathcal{H}}(\Sigma,\bm{\beta_G'},\bm{\alpha_G'},\bm{b},\bm{a},\bm{\overline{a}}),[-S],-g)$.
\end{lemma}

\begin{proof}
   It suffices to show that the statement is true after a single handleslide of one curve of $\bm{\alpha_G}$ ($\bm{\beta_G}$) over another curve $\bm{\alpha_G}$  ($\bm{\beta_G}$). In the $\bm{\alpha_G}$ case, one can construct an isomorphism by superimposing perturbations of $\bm{\alpha_G}$, $\bm{\alpha_G}'$ and $\bm{\beta_G}$ on $\Sigma$ to form a single Heegaard triple diagram. There is then an isomorphism given by counts of pseudo-holomorphic triangles as shown in Figure~\ref{fig:triangle}. Once again, the convenient positioning of the basepoints in neighborhoods of each of the points in $\bm{c}$ and $\bm{c}'$ implies the desired result.
\end{proof}

\begin{lemma}
Suppose $(\Sigma',\bm{\alpha_G'},\bm{\beta_G'})$ is obtained from $(\Sigma,\bm{\alpha_G},\bm{\beta_G})$ by (de)stabilization. If contact class like elements generate $\widehat{\HFK}(\widehat{\mathcal{H}}(\Sigma,\bm{\beta_G},\bm{\alpha_G},\bm{b},\bm{a},\bm{\overline{a}}),[-S],-g)$ then contact class like elements generate $\widehat{\HFK}(\widehat{\mathcal{H}}(\Sigma',\bm{\beta_G'},\bm{\alpha_G'},\bm{b},\bm{a},\bm{\overline{a}}),[-S],-g)$.
\end{lemma}
\begin{proof}
  We treat the stabilization case. Recall that one can think of stabilization of $\mathcal{H}(\Sigma,\bm{\alpha_G},\bm{\beta_G})$ as taking the connect sum of the canonical genus one Heegaard diagram for $S^3$, which contains a unique intersection point, $x$. The result is then immediate from the fact that the map induced by stabilization on $\widehat{\CFK}$ consists of appending the unique intersection point $x$ to all generators of $\widehat{\CFK}$. The destabilization case is similar.
\end{proof}

\begin{lemma}\label{lem:isotopiesinsp}
Suppose $\bm{a'}\cup\bm{\alpha_G'}\cup\bm{\beta_G'}$ are obtained from $\bm{a}\cup\bm{\alpha_G}\cup\bm{\beta_G}$ by isotopies in $\Sigma_P\cup\Sigma_G$. If contact class like elements generate $\widehat{\HFK}(\widehat{\mathcal{H}}(\Sigma,\bm{\beta_G},\bm{\alpha_G},\bm{b},\bm{a},\bm{\overline{a}}),[-S],-g)$ then contact class like elements generate $\widehat{\HFK}(\widehat{\mathcal{H}}(\Sigma,\bm{\beta_G'},\bm{\alpha_G'},\bm{b},\bm{a'},\bm{\overline{a}}),[-S],-g)$.
\end{lemma}

\begin{proof}
Once again this follows from considering the usual map on Heegaard Floer homology induced by isotopies, together with the convenient positioning of the basepoints near the $c_i$ intersection points.    
\end{proof}

We will prove stronger versions of Lemma~\ref{lem:isotopiesinsp} and Lemma~\ref{lem:arcovercurve} in Section~\ref{subsec:nicediagrams}.

\section{Generators of the next-to-minimal Alexander grading and Nice Diagrams}\label{sec:nexttominimal}

In this section we find elements of the next-to-bottom grading of knot Floer homology related to the elements of the bottom grading we found in the previous section, the contact class like elements. To show that these generators generate a summand, we need to control the induced differential $\partial_{\widehat{\CF}}^*$ on $\widehat{\CFK}(-K)$ --- i.e. the differential computing $\widehat{\HF}(Y)$, where $Y$ is the underlying $3$-manifold. To do this we appeal to a strategy due to Sarkar-Wang~\cite{MR2630063}, the details of which we will explain in Section~\ref{subsec:nicediagrams}. We then conclude the proof of the main Theorem, Theorem~\ref{thm:mainnoMaslov}, in Section~\ref{subsec:nexttobottom}.

\subsection{Nice Diagrams}\label{subsec:nicediagrams}

Recall that a Heegaard diagram $(\Sigma,\bm{\alpha},\bm{\beta})$ is \emph{nice} if every region of $\Sigma\setminus(\bm{\alpha}\cup\bm{\beta})$ either contains a basepoint, or is a bigon or a square~\cite[Definition 3.1]{MR2630063}. Equivalently, if we define the \emph{badness} of a $2n$-gon as $\max\{n-2,0\}$, a sutured Heegaard diagram is nice if every region that doesn't contain a boundary component has badness $0$. The key point of nice Heegaard diagrams is that computing their knot Floer homology is combinatorial. This will allow us to prove the main theorem by adapting a trick Baldwin-Vela-Vick used to prove the fibered case~\cite{baldwin_note_2018}. 

Informally, our goal in this subsection is to make our Heegaard diagrams nice while preserving the contact class like elements. Before proceeding we need to expand our definition of contact class like generators slightly. 
Consider a Heegaard diagram $\widehat{\mathcal{H}}(K)$, as defined in Section~\ref{subsubsec:HFhat}. We may pick closed curves $\ell_i$ such that:\begin{itemize}
  \item  each curve contains $z_1$,
  \item $\ell_i\cap({\bm{\alpha_S}}\cup\bm{\beta_S})=c_i$.
  \item $\ell_i\cap({\bm{\alpha_G}}\cup\bm{\beta_G})=c_i$.
\end{itemize}

Here $\bm{\alpha_S}$, and $\bm{\beta_S}$ are the curves introduced in Section~\ref{subsec:knotcomplementstosuturedmanifolds}. One can explicitly construct such a family of curves. Let $A$ denote a small neighborhood of $\underset{i}{\bigcup}\ell_i$.
We call isotopies supported in ${(\Sigma_G\cup\Sigma_P\cup\Sigma_S)\setminus A}$  \emph{contact class preserving isotopies}. Likewise we call handleslides of $\bm{\alpha}$ or  $\bm{\beta}$-curves over $\bm{\beta_G}$ or $\bm{\alpha_G}$-curves \emph{contact class preserving handleslides}. This terminology is justified by the fact that we have a well defined intersection point $c_i$ in any Heegaard diagram obtained from an adapted Heegaard diagram by a sequence of contact class preserving isotopies and contact class preserving handleslides. If $\mathcal{H}$ is a Heegaard diagram obtained from $\widehat{\mathcal{H}}(K)$ by contact class preserving isotopies and handleslides, we call an element $\bm{x}\in\widehat{\CFK}(\mathcal{H})$ a \emph{contact class like element} if it is a linear combination of generators each of which contains all of the distinguished intersection points $c_i$. This terminology is further justified by the following modest extension of Lemma~\ref{lem:isotopiesinsp} to this setting:

\begin{lemma}\label{lem:contactclasspreservingisotopies}
Suppose $\bm{\alpha_S'},\bm{\alpha'_G},\bm{\beta'_G},\bm{\beta_S'}$ are obtained from $\bm{\alpha_S},\bm{\alpha_G},\bm{\beta_G},\bm{\beta_S}$ respectively by contact class preserving isotopies. If contact class like elements generate $$\widehat{\HFK}(\mathcal{H}(\Sigma,\bm{\alpha_G},\bm{\beta_G},\bm{a},\bm{\overline{a}}),[-S],-g)$$ then contact class like elements generate $\widehat{\HFK}(\mathcal{H}(\Sigma,\bm{\alpha_G'},\bm{\beta_G'},\bm{a'}),[-S],-g)$.
\end{lemma}

\begin{proof}
Once again, this follows from considering the usual map on Heegaard Floer homology induced by isotopies, together with the convenient positioning of the basepoint $z_1$ in two of the diagonally opposite domains adjacent to $c_i$ for each $i$.    
\end{proof}

Likewise we have the following extension of Lemma~\ref{lem:arcovercurve}:

\begin{lemma}\label{lem:contactclasspreservinghandleslides}
Suppose $\bm{\alpha_S'}$, and $\bm{\beta_S'}$ are obtained from $\bm{\alpha_S},\bm{\alpha_G},\bm{\beta_G},\bm{\beta_S}$ respectively by contact class preserving handleslides. If contact class like elements generate $$\widehat{\HFK}(\mathcal{H}(\Sigma,\bm{\alpha_G},\bm{\beta_G},\bm{a},\bm{\overline{a}}),[-S],-g)$$ then contact class like elements generate $\widehat{\HFK}(\mathcal{H}(\Sigma,\bm{\alpha_G'},\bm{\beta_G'},\bm{a'}),[-S],-g)$.
\end{lemma}

\begin{proof}
Once again, this follows from considering the usual map on Heegaard Floer homology induced by isotopies, together with the convenient positioning of the basepoint $z_1$ in two of the diagonally opposite domains adjacent to $c_i$ for each $i$.    
\end{proof}

 We assume henceforth that $(\Sigma_G,\bm{\alpha_G},\bm{\beta_G})$ is nice and admissible, which we can do by~\cite[Theorem 6.4]{juhasz2008floer}.

\begin{lemma}\label{con:planardomains}
After a contact class preserving isotopy of the $\bm{\beta}_S$ curves we can assume that every region in $\widehat{\mathcal{H}}(K)$ without a basepoint is topologically a disk.  
\end{lemma}

Our proof is a more involved version of  that given in~\cite[Section 4.1, Step 1]{MR2630063}.

\begin{proof}
 Since $(\Sigma_G,\bm{\alpha_G},\bm{\beta_G})$ is nice, we may assume that if there is a region that is not topologically a disk then its boundary contains a sub-arc of $\bm{\beta}_S$. Thus, we can perform a finger move starting at $\beta_i\in\bm{\beta}_S$ along a non-separating arc in the domain and ending at some $\alpha$ boundary component of the region. We can always choose this finger move to be supported outside of $A$, since any domain which intersects $A$ non-trivially either contains $z_1$ or intersects $A$ in a small neighborhood of a vertex.
\end{proof}

We assume henceforth that $\widehat{\mathcal{H}}(K)$ satisfies the conclusion of Lemma~\ref{con:planardomains}.

\begin{definition}
   Consider the set of regions of $\widehat{\mathcal{H}}(K)$. A \emph{path} between regions $R_1$ and $R_n$ is a sequence of regions  $R_1,\dots R_n$ such that for all $i$: \begin{itemize}
       \item $R_i$ is not contained in the interior of $\Sigma_G$, and
       \item $R_i$ and $R_{i+1}$ have a common boundary component that is in $\bm{\beta}_S$ but not entirely contained in $\Sigma_G$.
   \end{itemize}
   The \emph{length} of a path is the number of regions it contains. The \emph{distance} between $R_1$ and $R_2$ is the minimal distance of a path between $R_1$ and $R_2$. If there is no path we say the distance is infinite. The \emph{distance of $R$}, $d(R)$, is the minimum distance of a path from $R$ to a region containing a basepoint. The \emph{distance of a Heegaard diagram $\mathcal{H}$}, $d(\mathcal{H})$, is the maximum distance of a bad region.
\end{definition}

\begin{lemma}\label{lem:distancefinite}
Every bad region of $\widehat{\mathcal{H}}(K)$ has finite distance.

\end{lemma}
 \begin{proof}
 
 We are assuming we have modified the Heegaard diagram as per Lemma~\ref{con:planardomains}. We are also assuming that the Heegaard diagram $(\Sigma_G,\bm{\alpha}_G,\bm{\beta}_G)$ is nice, and remains so after the addition of the $\bm{\beta}_S$ curves. It follows that any bad region $D$ intersects $\Sigma_P$ non-trivially. Observe that there is a path from  $D$ to a region containing a basepoint supported in $(\Sigma_P\cup\Sigma_S)\setminus A$, so the result follows. \end{proof}

Following~\cite{MR2630063}, we define the \emph{distance-$d$ complexity} of a finite-distance Heegaard diagram using our modified notion of distance.
That is, we set $$c_d(\mathcal{H}):=\Big(\underset{i=1}{\overset{m}{\sum}}b(D_i),-b(D_1),-b(D_2),\dots,-b(D_m)\Big),$$ where $D_1,D_2,\dots ,D_m$ are the distance $d$ bad regions with an arbitrary order subject to the condition that $b(D_1)\geq b(D_2)\geq\dots\geq b(D_m)$. We endow these tuples with the lexicographic ordering.

\begin{lemma} \label{con:nice}
Suppose we have applied the isotopies in Lemma~\ref{con:planardomains} to ensure that bad domains in $\widehat{\mathcal{H}}(K)$ are topologically disks. If ${c_d(\widehat{\mathcal{H}}(K))\neq 0}$ we may recursively modify $\widehat{\mathcal{H}}(K)$ by contact class preserving isotopies of $\widehat{\bm{\beta}}\setminus \bm{\beta}_G$ and contact class preserving handle-slides of curves in $\widehat{\bm{\beta}}\setminus \bm{\beta}_G$ to produce a finite sequence of Heegaard diagrams $\mathcal{H}_0=\widehat{\mathcal{H}}(K),\mathcal{H}_1, \mathcal{H}_2,\dots$ such that $d(\mathcal{H}_{i+1})\leq d(\mathcal{H}_i)$ and $c_d(\mathcal{H}_{i+1})<c_d(\mathcal{H}_i)$ for all $i$.
\end{lemma}

The idea of the proof is to imitate the proof of~\cite[Theorem 1.2]{MR2630063}, with the additional requirement that the isotopies and handle-slides can be chosen to be contact class preserving.

\begin{proof}
Note that we have modified $(\Sigma_G,\bm{\alpha}_G,\bm{\beta}_G)$ to be nice by appealing to~\cite[Theorem 6.4]{juhasz2008floer}, and indeed so that every domain in $\Sigma_G$ that is disjoint from the boundary is a rectangle or bigon even after adding in the curves $\bm{\beta_S}$.   Consider the maximum distance at which there are bad regions, $d$. Note this is non-zero by assumption. Consider $c_d(\mathcal{H}_0)$. Let $D_m$ denote the last bad region in the tuple $c_d(\mathcal{H}_0)$. It is a $2n$-gon with $n\geq 3$. Observe that $n$ is minimal among distance $d$ $2n$-gons. Let $b_*$ be an edge which separates $D_m$ from a region $D_*$ where $d(D_*)=d-1$. Note that $b_*\subset \bm{\beta_S}$. Let $a_1,\dots a_n$ be the $\alpha$-arcs oriented counter-clockwise from $b_*$.

        First we claim that there exists an arc $a_i$ such that the finger move through $a_i$ arcs and all subsequent $4$-gons does not return to $D_m$. Towards a contradiction assume this is not the case. By considering the union of all of the rectangles traversed by these finger moves, one can deduce that the $\beta$-curves are not homologically independent, contradicting the definition of a sutured Heegaard diagram. Thus applying a finger move through one of the series of rectangles does not terminate back at $D_m$, as claimed.

    Returning to the main proof, observe that performing a finger move of $b_*$ through an arbitrary $a_i$ arc and all subsequent squares terminates when one of the following occurs:\begin{enumerate}
        \item\label{item1} We reach a bigon.
        \item\label{item2} We reach a bad region of distance $\leq d-1$.
        \item\label{item3} We reach a bad region of distance $d$ that is not $D_m$.
         \item\label{item4} We return to $D_m$.
    \end{enumerate}

Indeed, notice that we can (and do) pick this finger move to be a contact class preserving isotopy.

 Suppose now that a finger moves through one of the arcs $a_i$ with $i\not\in\{1,n\}$ terminates at a region which is distinct from $D_m$ i.e. we are in Case~\ref{item1}, Case~\ref{item2}, or Case~\ref{item3}.

    Case~\ref{item1}. If a bigon is reached then observe that the total badness of distance $d$ decreases. Indeed, the distances of each of the bad regions does not change, so the distance of the Heegaard diagram does not change and we are done.

    Case~\ref{item2}. We reach a region, $D'$, of distance $\leq d-1$. Observe that if $D'$ is a bigon we are done by case 1). If $D'$ is not a bigon then after  finger move $D'$ is split into a bigon region and another region $D''$ with $b(D'')=b(D')+1$. On the other hand $D_m$ is split into two regions, the total badness of which is $b(D_m)-1$. That is, we see that the total badness stays the same. However, since, $d(D'')<d(D_m)$, the total badness at distance $d$ has decreased, as desired. It is also readily checked that the distance of the Heegaard diagram has not increased, since no bad regions of distance $>d$ were introduced.

    Case~\ref{item3}. As in the previous case, the total badness remains the same. However this time the total badness of distance $d$ remains the same. Nevertheless, we can note that the bad region of distance $d$ reached is $D_k$ for some $k<m$. Thus prior to reordering the tuple encoding the complexity, the last (mth) term increases by one, while the $k$th term decreases by $1$. This implies that the distance $d$ complexity decreases, as desired. Since no new bad regions of distance $>d$ were introduced, while one of the distance $d$ bad regions remained bad after the finger move, the distance remains the same.

    Suppose now that finger moves through $a_i$ terminates back at $D_m$ for all $i\neq 1,n$. We know that a finger moves through at least one of the $a_i$ (and all subsequent rectangles) does not terminate back at $D_m$. We will take it to be $a_1$. An identical argument applies in the case that it is $a_n$.

    We have two cases; $n=3$ or $n>3$.
    If $n=3$ we can then handle slide $\beta_*$ --- the $\beta$-curve containing $b_*$ --- over the $\bm{\beta_G}$ curve which yields the arc connecting $a_2$ and $a_3$. Note that this curve must lie in $\bm{\beta_G}$, rather than in $\bm{\beta_S}$, since if it did not, then the finger move would have to pass through the domain containing the $z_1$ basepoint, a contradiction. This is essentially identical to Case A in the proof of~\cite[Lemma 6.6]{juhasz2008floer}. This decreases the total badness at distance $d$ by one. Since no new bad regions of distance $>d$ are created, the maximum distance of a bad region does not increase, as required.

    If $n>3$ we split again as two cases according to whether or not the finger move through $a_2$ comes back via $a_i$ with $i<n$ or via $a_n$, just as in case B1 and B2 in the proof of~\cite[Lemma 6.6]{juhasz2008floer}.

    If the finger move comes back via $a_i$ with $i<n$ then instead of doing the finger move out through $a_2$ and back through $a_i$ we do the finger move out through $a_i$ returning to $D_m$ through $a_2$. We then continue the finger move out through $a_1$ until exiting the subsequent rectangles. See the right hand side of~\cite[Figure 7]{juhasz2008floer}. If the region this finger move terminates at is not also a bad region of distance $d$, then the total badness of distance $d$ decreases and we are done. On the other hand, if  the terminating region is bad of distance $d$, then it is of the form $D_i$ with $i<m$. Consequently, after performing the finger move, the distance $d$ complexity still decreases, just as in Case~\ref{item3}. Since the region that the finger move terminates at either remains bad or remains good, the maximum distance of a bad region does not increase, as required.
 
    Suppose finally that the finger move through $a_2$ returns to $D_m$ via $a_n$. Observe that there is a finger move of $b_*$ through some $a_p$ with $p\not\in\{1,2\}$ which returns to $D_m$ through $a_{n-1}$. Do this finger move, followed by a finger move through $a_n$ and back to $D_m$ through $a_2$ and finally out through $a_1$. See~\cite[Figure 8]{juhasz2008floer}.  If the region this finger move terminates at is not also a bad region of distance $d$, then the total badness of distance $d$ decreases, as desired. On the other hand, if the terminating region is bad of distance $d$, then it is of the form $D_i$ with $i<m$. Consequently, after performing the finger move, the distance $d$ complexity still decreases, just as in Case~\ref{item3}. Since the region that the finger move terminates at either remains bad or remains good, the maximum distance of a bad region does not increase, as required.\end{proof}

\begin{lemma}\label{lem:admissible}
   The terminal Heegaard diagram in the sequence in the statement of Lemma~\ref{con:nice} is admissible.
\end{lemma}

\begin{proof}
    We chose the Heegaard diagram $(\Sigma_G,\bm{\alpha}_G,\bm{\beta}_G)$ to be admissible. Admissibility now follows from the fact that of the four domains adjacent to $c_i$, a diagonally opposite pair contain $z_1$ and consequently must have multiplicity zero.
\end{proof}

\subsection{The next-to-bottom grading}\label{subsec:nexttobottom}

We can now prove the following more general version of the main theorem under weaker hypotheses.

\begin{theorem}\label{prop:generalrankbound}
    Suppose that $K$ is a non-trivial knot in a $3$-manifold $Y$. Let $S$ be a Seifert surface of $K$ such that we can construct a Heegaard diagram  as in Section~\ref{subsubsec:HFhat} with $\Sigma_P$ neither a disk nor an annulus. Let $V$ denote the summand of $\widehat{\HFK}(-K,-Y,[-S],-g(S))$ generated by contact class like generators. Then  $\widehat{\HFK}(-K,-Y,[-S],1-g(S))$ contains a $V[1]$ summand.

    \end{theorem}
    Here $[1]$ denotes an upward shift in the Maslov grading by one. We do not require $S$ to be minimal genus, though in the case that it is not, the statement is vacuously true, since $V=0$ by the $3$-dimensional adjunction inequality.

\begin{proof}[Proof of Theorem~\ref{prop:generalrankbound}]
Consider a Heegaard diagram $\widehat{\mathcal{H}}(-K_\psi)$ for $(-K_\psi, -Y)$ as in Section~\ref{subsubsec:HFhat}, where we impose the additional requirement that $\alpha_1\cap\Sigma_G=\emptyset$. The last condition is guaranteed by the technical hypothesis of the theorem.  By applications of Lemma~\ref{con:nice} and~\ref{lem:admissible} we can obtain a nice admissible Heegaard diagram for $\mathcal{H}$ with the property that $\alpha_1$ is disjoint from all of the $\bm{\beta}_G$-curves. 

We first treat the case in which $\alpha_1$ and $\beta_1$ are not isotopic. We can assume without loss of generality that $\beta_1$ is strictly to the left of $\alpha_1$ in the sense that after isotoping $\alpha_1$ and $\beta_1$ so that they agree in $\Sigma_S$ and intersect minimally in $\Sigma_G\cup\Sigma_P$, $\beta_1$ is strictly to the left of $\alpha_1$ in $\Sigma$. The symmetry properties of knot Floer homology (See~\cite[Section 3]{Holomorphicdisksandknotinvariants}) take care of the case in which $\alpha_1$ is not sent strictly to the left. Specifically we have that 
$$\widehat{\HFK}_m(-Y,-K,[-S],A)\cong \widehat{\HFK}_{-m}(Y,-K,[-S],-A)\cong \widehat{\HFK}_{-2A-m}(Y,-K,[-S],A)$$
and that if $(-Y,-K)$ admits a Heegaard diagram in which $\beta_1$ is to the right of $\alpha_1$, then $(Y,-K)$ admits a Heegaard diagram in which $\beta_1$ is strictly to the left of $\alpha_1$.

Consider the sutured Heegaard diagram $\mathcal{H}'=(\Sigma_S\cup\Sigma_P\cup\Sigma_G,\alpha_1,\beta_1)$ obtained from $\mathcal{H}$ by removing all of the $\alpha$ and $\beta$-curves save for $\alpha_1$ and $\beta_1$. Let $(Y,\gamma)$ be the corresponding sutured manifold. Consider $\CF(\Sigma_S\cup\Sigma_P\cup\Sigma_G\setminus \nu(z_1),\alpha_1,\beta_1)$. This comes endowed with a filtration defined by $\mathcal{F}(x)=\langle c_1(\s(x),t),[\Sigma_S]\rangle$.

We claim that there are elements $x_i\in\alpha_1\cap\beta_1$ such that $\partial_{\CF}\sum x_i=c_1$. To verify this observe that since we have assumed without loss of generality that after pulling tight $\beta_1$ lies to the right of $\alpha_1$ after removing excess bigons via contact class preserving isotopies we must have the claim in this setting. Observe that contact class preserving isotopies preserve the class $[c_1]\in \CF(\Sigma_S\cup\Sigma_P\cup\Sigma_G\setminus (\nu(z_1)\cup\nu(w_{4g})),\alpha_1,\beta_1)$ as in the proof of Lemma~\ref{lem:contactclasspreservingisotopies}. Note that $\partial_{\CF}(x_i)$ can be computed purely as a count of bigons.

Consider now the set of contact class like elements, $\bm{c}(\bm{x})$, in $\widehat{\CFK}(\widehat{\mathcal{H}}(K_\psi))$. Recall that these are linear combinations of generators obtained by taking linear combinations of generators $\bm{x}\in\T_{\bm{\alpha_G}}\cap\T_{\bm{\beta_G}}$ and appending the intersection points $c_i$ to each generator. Consider too the generators $\bm{d}_j$ that are obtained by replacing each instance of $c_1$ in a contact class like element with the element $x_j$. Consider $\partial_{\widehat{\CF}}{\bm{d}_j}$. Observe that since the diagram is nice, the differential only counts bigons and squares. These are constant on $c_i$ for all $i\neq 1$. Indeed, since $\alpha_1$ does not intersect any of the $\bm{\beta_G}$-curves we see that $$\partial_{\widehat{\CF}}\sum_{j}\bm{d}_j=\bm{c}(\bm{x})+\sum_j(\partial_{\CF(\mathcal{H}(G))} \bm{x})\cup\{x_j\}\cup\{c_i:i\neq 1\}.$$ In particular, $\partial_{\widehat{\CF}}\sum_j\bm{d}_j=\bm{c}(\bm{x})$ for contact class like cycles, so that $\partial_{\CF}^*$ is a surjection onto $V$. Since $\partial_{\widehat{\CF}}^*$ lowers the Maslov grading by one, the result follows.

We now proceed to the case in which $\alpha_1$ and $\beta_1$ are isotopic. Observe that in this case $Y$ has an $S^1\times S^2$ summand. Consider $\xi$ the Poincar\'e dual of $S^2$. The argument now follows exactly as in the case that $\beta_1$ is strictly to the left of $\alpha_1$, except instead of computing $\partial_{\widehat{\CF}}^*$ we compute $A_\xi^1$, where $A_\xi^1:\widehat{\HFK}(-K,-Y,[-S],1-g(S))\to \widehat{\HFK}(-K,-Y,[-S],-g(S))$ is the homological action of $\xi$ on knot Floer homology. This can be computed using the green curve representing $\xi$ shown in Figure~\ref{fig:protoHD}, and decreases the Maslov grading by one. See~\cite[Section 9]{Holomorphicdisksandknotinvariants} for details. In particular $A_\xi^1$ factors through a surjection $$\widehat{\HFK}(-K,-Y,[-S],1-g(S))\surj V,$$ concluding the proof. \end{proof}

We can now prove the main theorem;

\begin{proof}[Proof of Theorem~\ref{thm:mainnoMaslov}]
    By Proposition~\ref{prop:satiatedimpliesgeneration} we have that $\widehat{\HFK}(-K,-Y,[-\Sigma]-g)$ is generated by contact class like generators. The result now follows from Theorem~\ref{prop:generalrankbound} and the symmetry properties of knot Floer homology.
\end{proof}

\bibliographystyle{alpha}
\bibliography{bibliography}
\end{document}